\documentclass[11pt]{article}

\newcommand{\later}[1]{{}}
\newcommand{\old}[1]{{}}
\long\def\ignore#1{}
\def\cut{\ignore}

\cut{25 figure files:
  lb_grid.pdf f27.pdf f28.pdf f7.pdf f8.pdf
  f23.pdf f24.pdf f25.pdf f26.pdf f10.pdf
  f11.pdf f12.pdf f13.pdf f14.pdf f15.pdf
  f16.pdf f19.pdf f18.pdf f20.pdf f21.pdf
  f22.pdf f29.pdf f30.pdf f31.pdf latticeHspanner.pdf
} 

\usepackage{amsmath}
\usepackage{amsthm}
\usepackage{amssymb}
\usepackage{amsfonts}
\usepackage{float}
\usepackage{graphicx}
\usepackage[small]{caption}
\usepackage{color}
\usepackage{enumitem}
\usepackage{marvosym}
\usepackage{cases}
\usepackage{url}

\newtheorem{theorem}{Theorem}

\newtheorem{conjecture}{Conjecture}

\newtheorem{fact}{Fact}

\newcommand{\NN}{\mathbf{N}} 
\newcommand{\RR}{\mathbf{R}} 

\newcommand{\etal}{{et~al.}}
\newcommand{\ie}{{i.e.}}

\def\deg{\texttt{deg}}

\def\path{\texttt{path}}

\begin{document}

\title{Lattice spanners of low degree\footnote{%
A preliminary version in: Proceedings of the 2nd International Conference on Algorithms
and Discrete Applied Mathematics, Thiruvananthapuram, India, 
Feb. 2016, vol $9602$ of LNCS.}}

\author{%
Adrian Dumitrescu\\
\small Department of Computer Science\\[-0.8ex]
\small University of Wisconsin--Milwaukee\\[-0.8ex]
\small Milwaukee, WI, USA\\
\small\tt \Email{ dumitres@uwm.edu}\\
\and
Anirban Ghosh\\
\small Department of Computer Science\\[-0.8ex]
\small University of Wisconsin--Milwaukee\\[-0.8ex]
\small Milwaukee, WI, USA\\
\small\tt \Email{ anirban@uwm.edu}\\
}

\date{\today}
\maketitle

\begin{abstract}
  Let $\delta_0(P,k)$ denote the degree $k$ dilation of a point set~$P$
 in the domain of plane geometric spanners.
If $\Lambda$ is the infinite square lattice, it is shown that 
$1+\sqrt{2} \leq \delta_0(\Lambda,3) \leq (3+2\sqrt2) \, 5^{-1/2} = 2.6065\ldots$
and $\delta_0(\Lambda,4) = \sqrt{2}$.
If $\Lambda$ is the infinite hexagonal lattice, it is shown that 
  $\delta_0(\Lambda,3) = 1+\sqrt{3}$ and $\delta_0(\Lambda,4) = 2$. 
All our constructions are planar lattice tilings constrained to degree
$3$ or $4$.

\end{abstract}

\section{Introduction} \label{sec:intro}

Let $P$ be a (possibly infinite) set of points in the Euclidean plane.
A \emph{geometric graph} embedded on $P$ is a graph $G = (V,E)$ where
$V = P$ and an edge $uv \in E$ is the line segment connecting $u$ and $v$. 
View $G$ as a edge-weighted graph, where the weight of $uv$ is 
the Euclidean distance between $u$ and $v$.
A geometric graph $G$ is a \emph{t-spanner}, for some $t \geq 1$,
if for every pair of vertices $u,v$ in $V$, the length of the shortest
path $\pi_G(u,v)$ between $u$ and $v$ in $G$ is at most $t$ times $|uv|$,
\ie, $\forall u,v \in V, |\pi_G(u,v)| \leq t |uv|$.
Obviously, the complete geometric graph on a set of points is a
1-spanner. When there is no need to specify $t$, the rather imprecise term
\emph{geometric spanner} is also used. 
A geometric spanner $G$ is \emph{plane} if no two edges in $G$ cross.
Here we only consider plane geometric spanners.
A geometric spanner of degree at most $k$ is called
\emph{degree $k$ geometric spanner}. 

Consider a geometric spanner $G=(V,E)$. The \emph{vertex dilation}
or \emph{stretch factor} of a pair $u,v\in V$, denoted $\delta_G(u,v)$,
is defined as $\delta_G(u,v) =  |\pi_G(u,v)| / |uv|$.
If $G$ is clear from the context, we simply write $\delta(u,v)$. 
The \emph{vertex dilation} or \emph{stretch factor} of $G$,
denoted $\delta(G)$,  is defined as  
$\delta(G) = \sup_{u,v \in V} \delta_G(u,v). $ The terms \emph{graph theoretic dilation}
and \emph{spanning ratio} are also used~\cite{EGK06,KG92,NS07}. 

Given a point set $P$, let the \emph{dilation} of $P$, denoted by $\delta_0(P)$,
be the minimum stretch factor of a plane geometric graph (equivalently, triangulation)
on vertex set $P$; see~\cite{Mu04}. 
Similarly, let the \emph{degree $k$ dilation} of $P$, denoted by $\delta_0(P,k)$,
be the minimum stretch factor of a plane geometric graph of degree at most $k$
on vertex set $P$. Clearly, $\delta_0(P,k) \geq \delta_0(P)$
holds for any $k$. Furthermore, $\delta_0(P,j) \geq \delta_0(P,k)$ holds for any $j<k$.
(Note that the term \emph{dilation} has been also used with different meanings in
the literature, see for instance~\cite{PBMS13,KKP15}.)

The field of geometric spanners has witnessed a great deal of interest from researchers,
both in theory and applications;
see for instance the survey articles~\cite{PBMS13,Epp00,GK07,NS07}.
For the current status of various open problems in this area, the reader is referred
to the web-page maintained by Smid~\cite{Smid-open}. 

Typical objectives include constructions of low stretch factor geometric spanners
that have few edges, bounded degree, low weight and/or diameter, etc. 
Geometric spanners find their applications in the areas of robotics,
computer networks, distributed systems and many others.
Various algorithmic and structural results on sparse geometric spanners can be 
found in~\cite{AKK+08,ADD+93,ABC+08,CDNS95,CHL08,EKLL04,KKP15,LL92}.

Chew~\cite{Ch89} was the first to show that it is always possible to construct
a plane 2-spanner with $O(n)$ edges on a set of $n$ points; more recently,
Xia~\cite{Xia13} proved a slightly sharper upper bound of $1.998$ 
using Delaunay triangulations.
Bose~\etal~\cite{BGS05} showed that there exists a plane $t$-spanner
of degree at most 27 on any set of points in the Euclidean plane
where $t \approx 10.02$. The result was subsequently improved
in~\cite{BGHP10,BSX09,BCC12,KP08,LW04} in terms of degree. Recently,
Bonichon~\etal~\cite{BKPX15} reduced the degree to $4$ with $t \approx 156.82$. 
The question whether the degree can be reduced to 3 remains open
at the time of this writing; if one does not insist on having a plane spanner,
Das~\etal~\cite{DH96} showed that degree 3 is achievable.
From the other direction, lower bounds on the stretch factors of plane spanners
for finite point sets have been investigated in~\cite{DG15a,KKP15,Mu04}. 

It is natural to study the existence of low-degree spanners of
fundamental regular structures, such as point lattices. Indeed, these have
been the focus of interest since the early days of computing.
One such intense research area concerns VLSI~\cite{Le83}. 
Other applications of spanners (not necessarily geometric) are
in the areas of computer networks and parallel computing;
see for instance~\cite{LS93,LSS96}. While the authors of~\cite{LS93,LSS96}
do examine grid structures (including planar ones), the resulting stretch factors
however are not defined (or measured) in geometric terms.
More recently, lattice structures at a larger scale are used in industrial design,
modern urban design and outer space design. 
Indeed, Manhattan-like layout of facilities and road connections
are very convenient to plan and deploy, frequently in an automatic manner.
Studying the stretch factors that can be achieved in low degree spanners
of point sets with a lattice structure appears to be quite useful.
The two most common lattices are the square lattice and the hexagonal lattice. 

According to an argument due to Das and Heffernan~\cite{DH96},\cite[p.~468]{NS07},
the $n$ points in a $\sqrt{n} \times \sqrt{n}$ section of the integer lattice cannot be
connected in a path or cycle with stretch factor $o(\sqrt{n})$, $O(1)$ in particular.
Similarly, no degree $2$ plane spanner of the infinite integer lattice
can have stretch factor $O(1)$, 
hence a minimum degree of $3$ is necessary in achieving a constant stretch factor. 
The same facts hold for the infinite hexagonal lattice.

\paragraph{Our results.}
Let $\Lambda$ be the infinite square lattice.
  We show that the  degree $3$ and $4$ dilation of this lattice are bounded as follows:
  
\begin{itemize}
 \item[(i)]~$1 + \sqrt{2} \leq \delta_0(\Lambda,3) \leq (3+2\sqrt2) \, 5^{-1/2}$ 
(Theorem~\ref{thm:square:3}, Section~\ref{sec:square}).
  \item[(ii)]~$\delta_0 (\Lambda,4)  = \sqrt{2}$
(Theorem~\ref{thm:square:4}, Section~\ref{sec:square}). 
\end{itemize}

 If $\Lambda$ is the infinite hexagonal lattice, we show that
 
\begin{itemize}
\item[(i)] $\delta_0(\Lambda,3) = 1+\sqrt{3}$
   (Theorem~\ref{thm:hexagonal:3}, Section~\ref{sec:hexagonal}).
\item[(ii)] $\delta_0(\Lambda,4) = 2$
   (Theorem~\ref{thm:hexagonal:4}, Section~\ref{sec:hexagonal}).
\end{itemize}

\section{Preliminaries} \label{sec:prelim}

By the well known Cauchy-–Schwarz inequality for $n=2$,
if $a,b,x,y \in \RR^+$, then
$$ g(x,y) = \frac{ax+by}{\sqrt{x^2+y^2}} \leq \sqrt{a^2+b^2}, $$
and moreover, $g(x,y) = \sqrt{a^2+b^2}$ when $x/y = a/b$.
In this paper, we will use this inequality in an equivalent form:

\begin{fact} \label{fact1}
        Let $a,b,\lambda \in \RR^+$. Then 
$ f(\lambda) = \dfrac{a \lambda +b}{\sqrt{\lambda^2+1}} \leq \sqrt{a^2+b^2}$, 
and moreover, $f(\lambda) = \sqrt{a^2+b^2}$ when $\lambda = a/b$.
\end{fact}

\paragraph{Notations and assumptions.}
Let $P$ be a planar point set and $G=(V,E)$ be a plane geometric graph
on vertex set $P$. For $p,q \in P$, 
$pq$ denotes the connecting segment and $|pq|$ denotes its Euclidean length.
The degree of a vertex (point) $p \in V$ is denoted by $\deg(p)$.
For a specific point set $P=\{p_1,\ldots,p_n\}$,
we denote the shortest path between $p_{s},p_{t}$ in $G$ consisting of
vertices in the order 
$p_{s},\ldots,p_{t}$ using $\rho(p_{s},\ldots,p_{t})$ and by
$|\rho(p_{s},\ldots,p_{t})|$ its total Euclidean length.
The graphs we construct have the property that no edge
contains a point of $P$ in its interior.

\section{The square lattice} \label{sec:square}

This section is devoted to the degree 3 and 4 dilation of the square
lattice. In~\cite{DG15b}, we showed that the degree 3 dilation of the
infinite square lattice is at most
$(7+5\sqrt{2})\,29^{-1/2}=2.6129\ldots$ Here we improve this upper
bound to $\delta_0 := (3+2\sqrt2) \, 5^{-1/2}=2.6065\ldots$ We
believe that this upper bound is the best possible, and so in this
section we present two degree $3$ spanners for the infinite square
lattice that attain this bound. Another possible candidate is
  presented in Section~\ref{sec:remarks}.  
\begin{theorem} \label{thm:square:3}
Let $\Lambda$ be the infinite square lattice. Then,
$$2.4142\ldots= 1+\sqrt{2} \leq \delta_0(\Lambda,3) \leq (2\sqrt{2}
+3) \, 5^{-1/2}=2.6065\ldots $$ 
\end{theorem}
\begin{proof}
 To prove the lower bound, consider any point $p_0 \in \Lambda$ and its
eight neighbors $p_1,\ldots,p_8$, as in Fig.~\ref{fig:lb_grid}. 
	\begin{figure}[hbtp]
		\centering
		\includegraphics[scale=0.42]{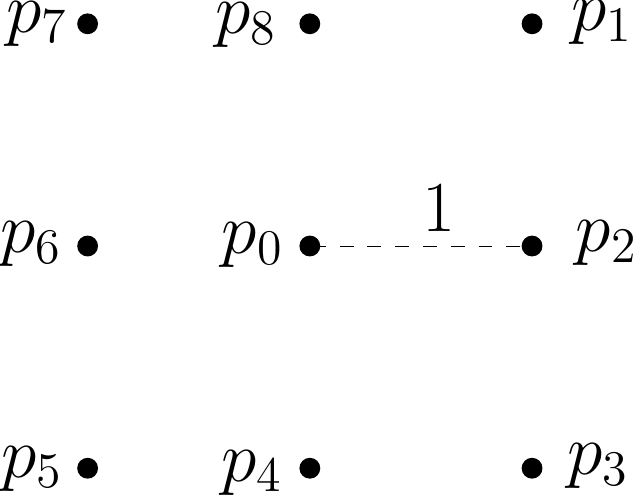}
		\caption{Illustrating the lower bound of $1 +
                  \sqrt{2}$ for the square lattice.} 
		\label{fig:lb_grid}
	\end{figure}
Since $\deg(p_0) \leq 3$, $p_0$
  can be connected to at most three neighbors from $\{p_2,p_4,p_6,p_8\}$.
 We may assume that the edge $p_0p_2$ is not present; then 
	\begin{equation*}
	  \delta(p_0,p_2) \geq \frac{|\rho(p_0,p_i,p_2)|}{|p_0p_2|} \geq 1 + \sqrt{2},
          \text{ where } i\in\{1,3,4,8\}.
	\end{equation*}

To prove the upper bound, we construct a plane degree 3 geometric
graph $G$ as illustrated in Fig.~\ref{fig:f27} (left); observe that there are four
types of vertices in $G$. For any two lattice points $p,q \in
\Lambda$, we construct a path in $G$. 
Set $p=(0,0)$ as the origin and consider the four quadrants 
$W_i$, $i=1,\ldots,4$, labeled counterclockwise in the standard
fashion; see Fig.~\ref{fig:f27} (right).
Points on the dividing lines are assigned arbitrarily to any of the two
adjacent quadrants.
By the symmetry of $G$, we can assume that $q$ lies in the first
quadrant, thus $q=(x,y)$, where $x,y \geq 0$, while the origin $p=(0,0)$ can
be at any of the four possible types of lattice points.

Consider the path from $p=(0,0)$ to $q=(x,y)$ via $(z,z)$, where $z=\min(x,y)$,
that visits every other lattice point on this diagonal segment as
shown in Fig.~\ref{fig:f6}, and let $\ell(x,y)$ denote its length.
If $x=0$, the stretch factor is easily seen to be at most $1+\sqrt2$. 
Since a horizontal path connecting two points with the same
$y$-coordinate at distance $a$ is always shorter 
than any path connecting two points with the same
$x$-coordinate at the same distance $a$, it is enough to prove our
bound on the stretch factor in the case $y \geq x$ (\ie, $z=x$).  
We thus subsequently assume that $y \geq x \geq 1$. 
\begin{figure}[ht]
	\centering
	\includegraphics[scale=0.42]{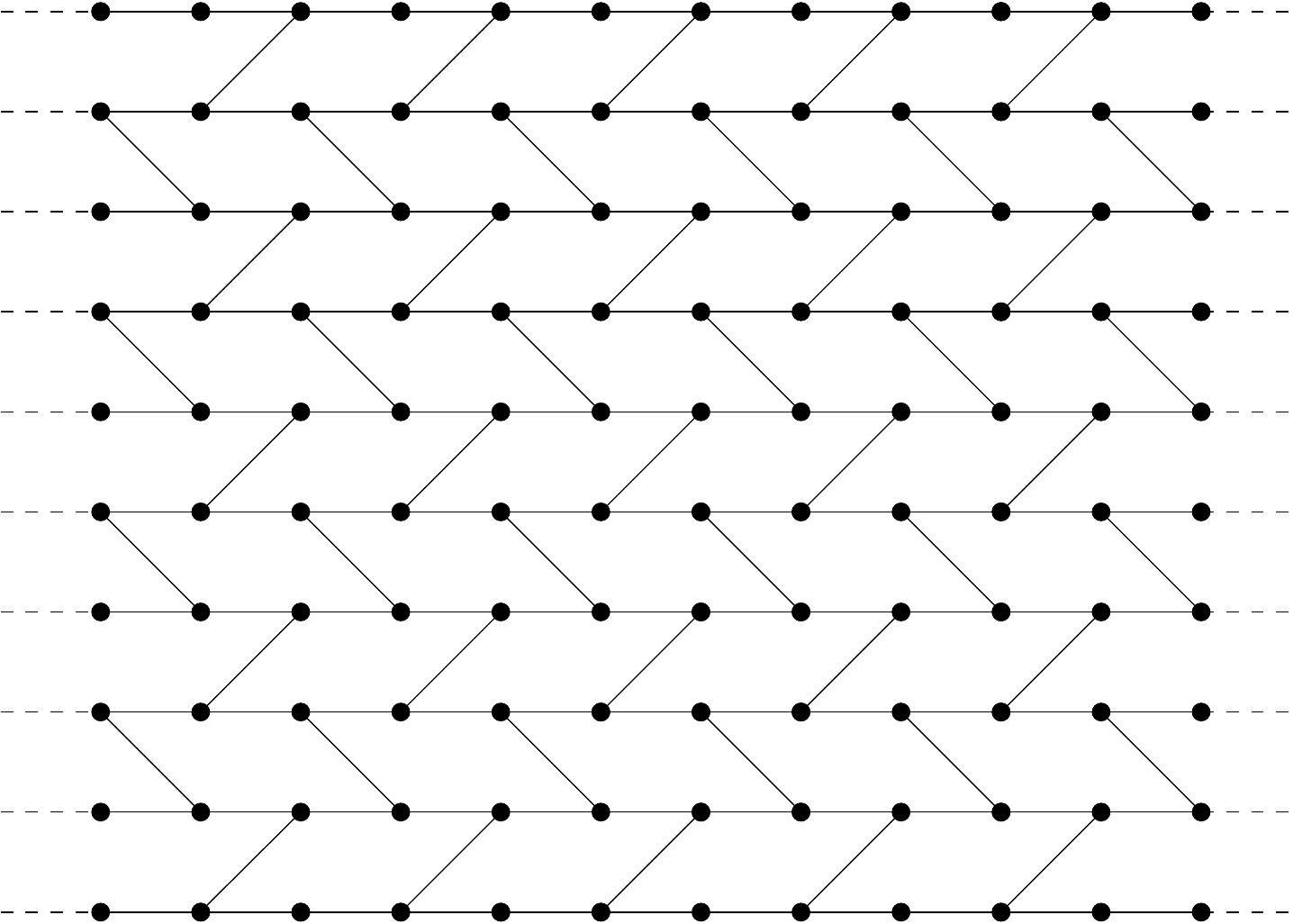}
	\hspace{10mm}
	\includegraphics[scale=0.5]{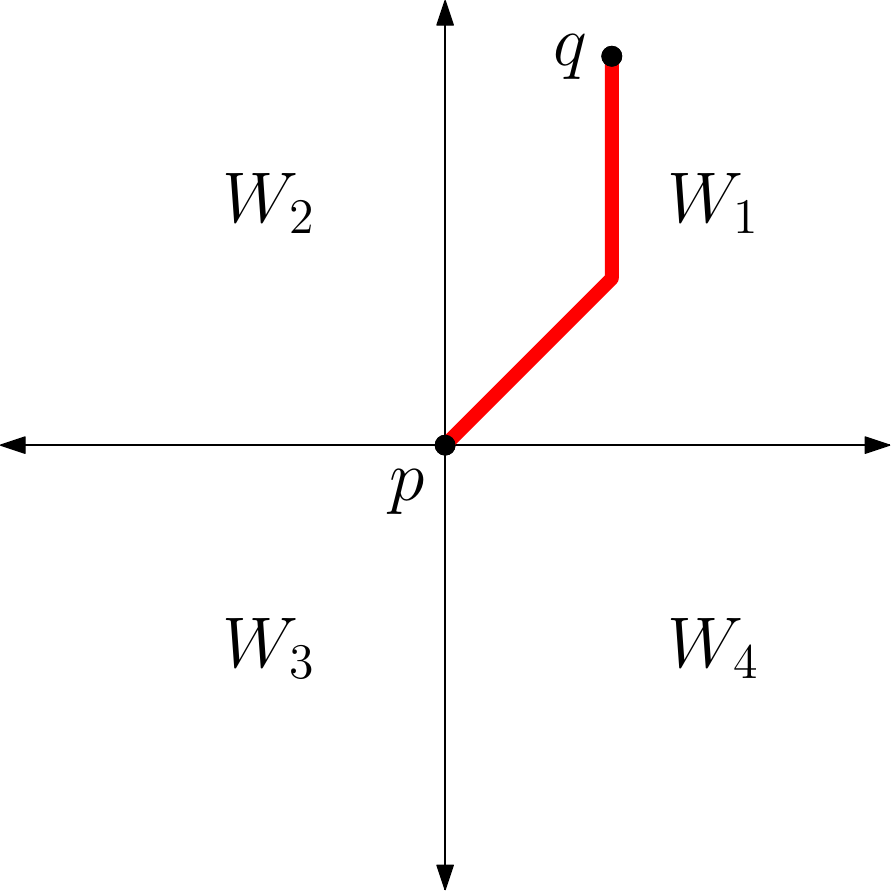}
	\caption{Left: a degree $3$ plane graph on $\Lambda$. Right: a
          schematic diagram showing the path between $p,q$ (when $x \leq y$).
          The bold path consist of segments of lengths $1$ and $\sqrt{2}$.} 
	\label{fig:f27}
\end{figure}

\begin{figure}[ht]
	\centering
	\includegraphics[scale=0.42]{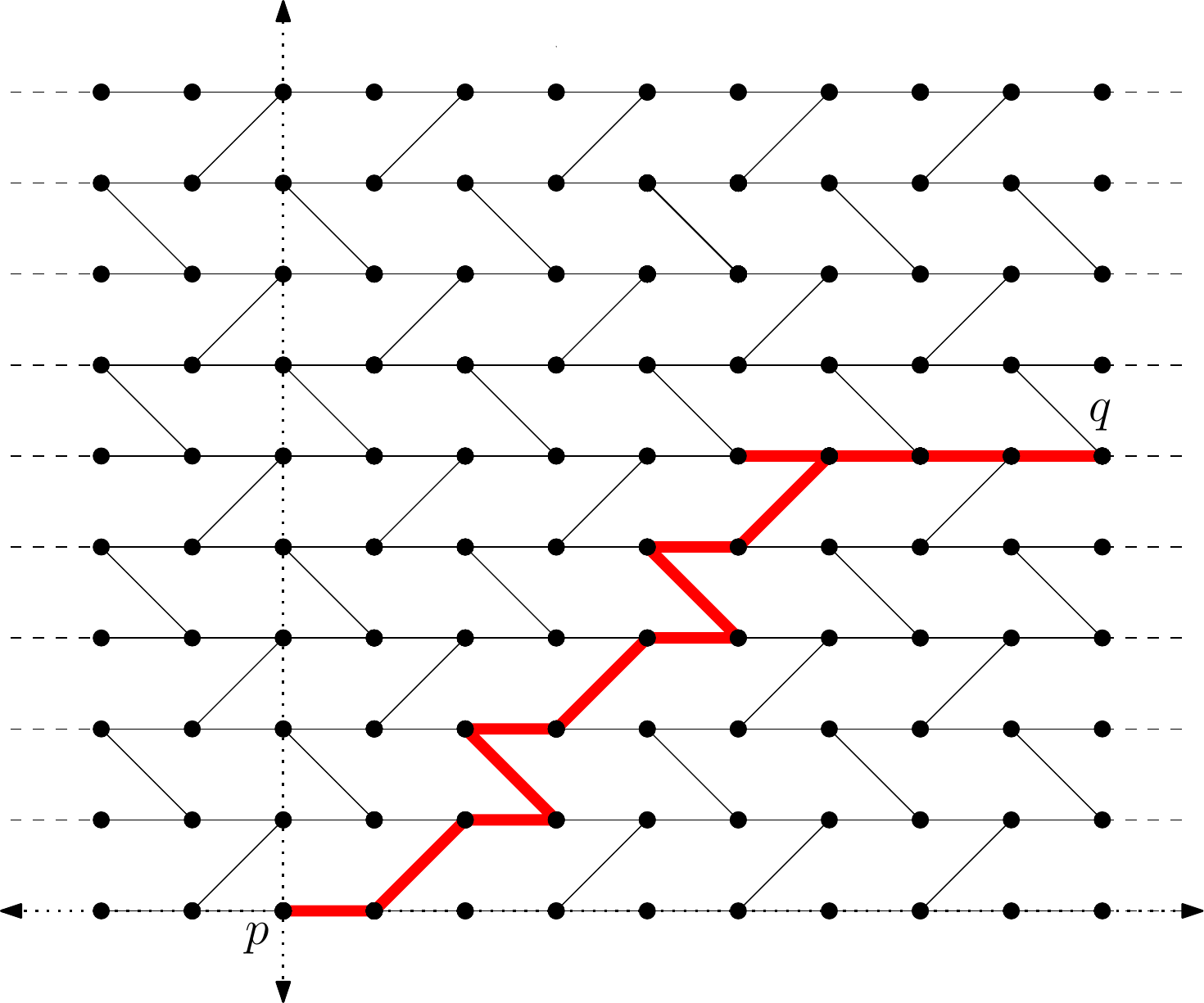}
	\hspace{10mm}
	\includegraphics[scale=0.42]{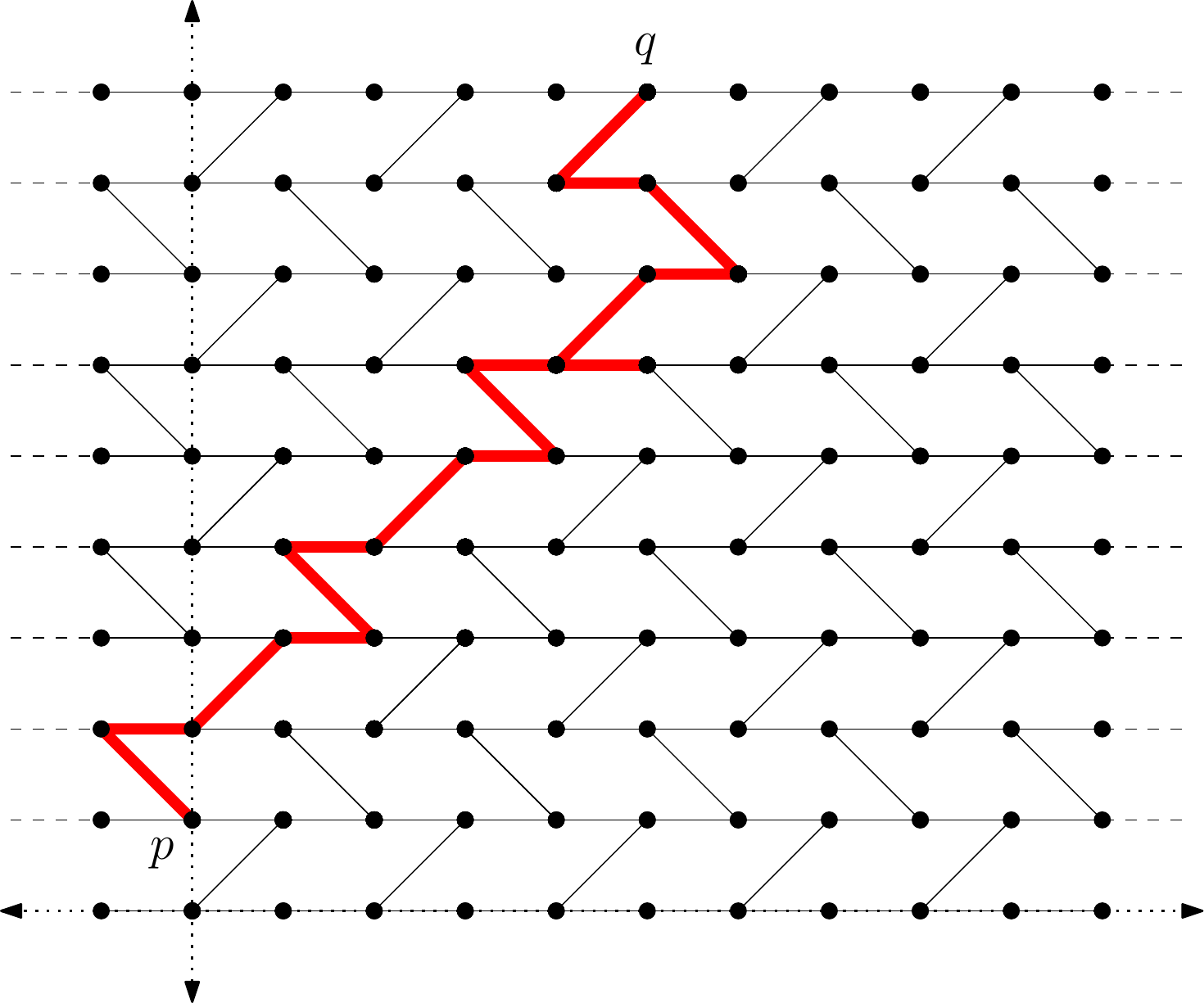}
	\caption{Paths connecting $p$ to $q$ in $G$ generated by the procedure
          outlined in the text. Observe that in both examples a unit horizontal edge
          is traversed in both directions (but can be shortcut).}
	\label{fig:f6}
\end{figure}

Observe that connecting points $(a,a)$ with $(a+2,a+2)$, for any $a \geq 0$, 
requires length $2+2\sqrt2$, and that
connecting points $(a,a)$ with $(a+1,a+1)$, for any $a \geq 0$, 
requires length at most $2+\sqrt2$. It follows that
\begin{align*} 
	\ell(x,y) &\leq \left( 2 \Big \lceil \frac{x}{2} \Big \rceil + \sqrt2 x \right) 
	+ (y-x) (1+\sqrt{2}) \nonumber \\
	&\leq 2 \left( \frac{x+1}{2} \right) + \sqrt{2} x + (y-x) (1+\sqrt{2})
	= 1+y(1+\sqrt{2}).
\end{align*}

Since $|pq|=\sqrt{x^2 + y^2}$,
the corresponding stretch factor is bounded in terms of $x,y$ as follows
\begin{equation} \label{eq:12}
	\delta(p,q) \leq \gamma(x,y) :=\frac{1+y(1+\sqrt{2})}{\sqrt{x^2 + y^2}}. 
\end{equation}

We now consider the case $x=1$ separately. Let $\lambda=\frac{1}{y}$, 
where $y=1,2,3,4,5,\ldots$, and so
$\lambda=1, \frac12, \frac13,  \frac14,  \frac15,\ldots \in (0,1)$. 
According to \eqref{eq:12} we have 
\begin{equation*} 
	\gamma(1,y) \leq \frac{1+y(1+\sqrt{2})}{\sqrt{y^2 +1}}
	= \frac{\lambda + 1+\sqrt{2} }{\sqrt{\lambda^2 +1}} =: f(\lambda). 
\end{equation*}
The derivative $f'$ vanishes at $\lambda_0= \frac{1}{\sqrt2 +1} = \sqrt2 -1= 0.4142\ldots$
On the interval $(0,1)$: $f$ is increasing on the interval $(0,\lambda_0)$ and
decreasing on the interval $(\lambda_0,1)$; 
it attains a unique maximum at $\lambda=\lambda_0$. 
Since $\lambda_0 \in (\frac13, \frac12)$, we have
$$ f(\lambda) \leq \max \left( f\left(\frac13\right), f\left(\frac12\right) \right)=
f\left(\frac13\right) = f\left(\frac12\right)= \delta_0 .$$

It remains to consider the case $x \geq 2$; according to \eqref{eq:12} we have
\begin{align*} 
	\delta(p,q) &\leq \frac{1+y(1+\sqrt{2})}{\sqrt{x^2+y^2}}
	\leq \frac{1+y(1+\sqrt{2})}{\sqrt{4+y^2}}\\
	&= \frac{(1+\sqrt{2})(y/2) +1/2}{\sqrt{(y/2)^2 +1}}
	\leq \sqrt{(1+\sqrt{2})^2 + 1/4} = 2.4654\ldots < \delta_0,
\end{align*}
where the last inequality follows from Fact~\ref{fact1} by setting $\lambda=y/2$.

This completes the case analysis. Observe that the above analysis is
tight since there are point pairs with $x=1,y=2$ having pairwise
stretch factor $\delta_0$. 
We have thus shown that for any $p,q \in \Lambda$,
we have $\delta(p,q) \leq (3+2\sqrt2) \, 5^{-1/2}$,
completing the proof of the upper bound, and thereby the proof of
Theorem~\ref{thm:square:3}. 
\end{proof}

\paragraph{Another degree $3$ spanner with stretch factor $\delta_0 =
  (3+2\sqrt2)\,5^{-1/2}$.}
The graph $G$ is illustrated in Fig.~\ref{fig:23} (left). For any two
lattice points $p,q \in \Lambda$, we construct a path in $G$. 
Set $p=(0,0)$ as the origin and consider the four quadrants 
$W_i$, $i=1,\ldots,4$, labeled counterclockwise in the standard
fashion; see Fig.~\ref{fig:23} (right). Points on the dividing lines
are assigned arbitrarily to any of the two adjacent quadrants. By the
symmetry of $G$, we can assume that $q$ lies in one of the first two
quadrants. 
\begin{figure}[ht]
	\begin{center}
		\includegraphics[scale=0.42]{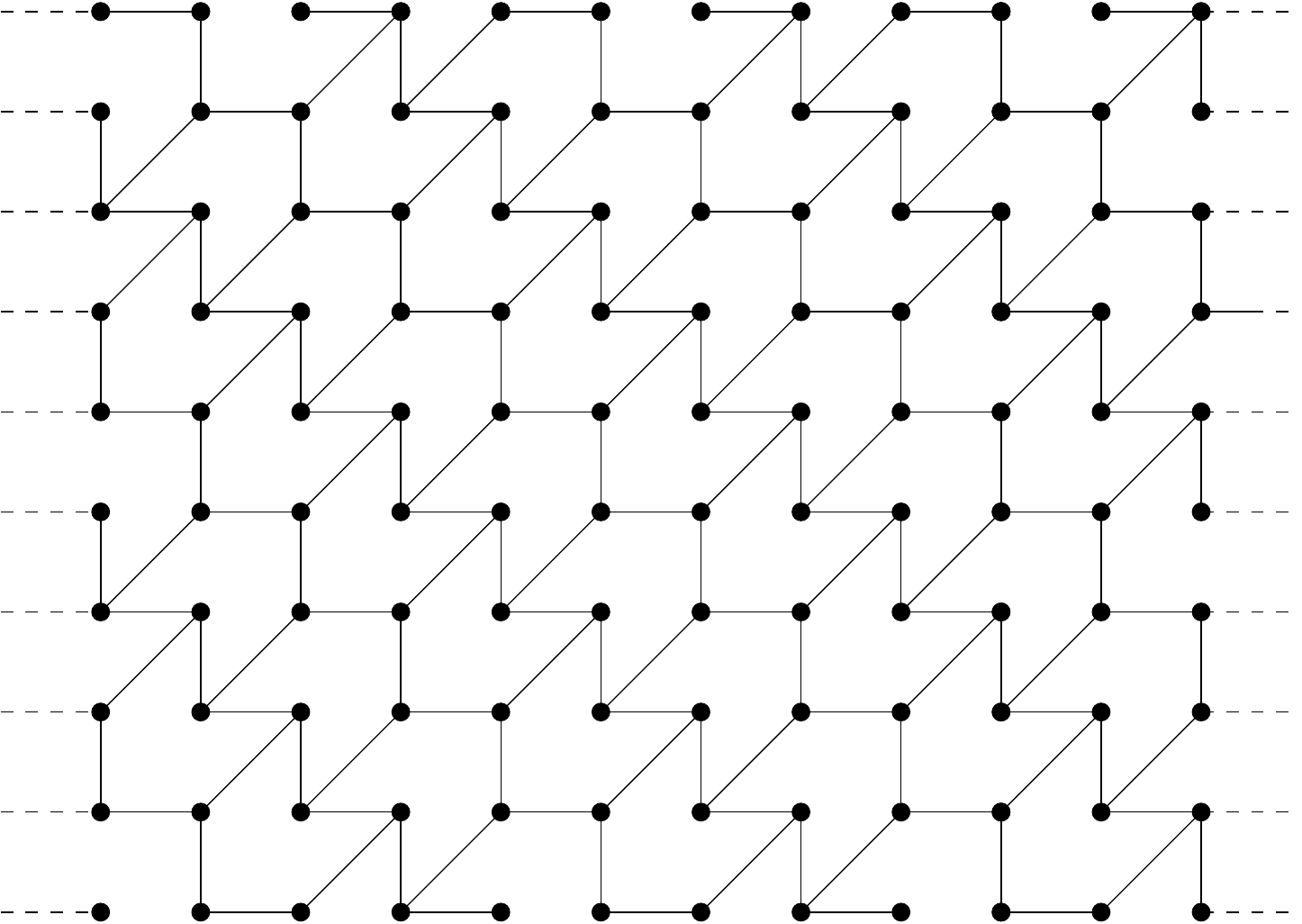}
		\hspace{10mm}
		\includegraphics[scale=0.42]{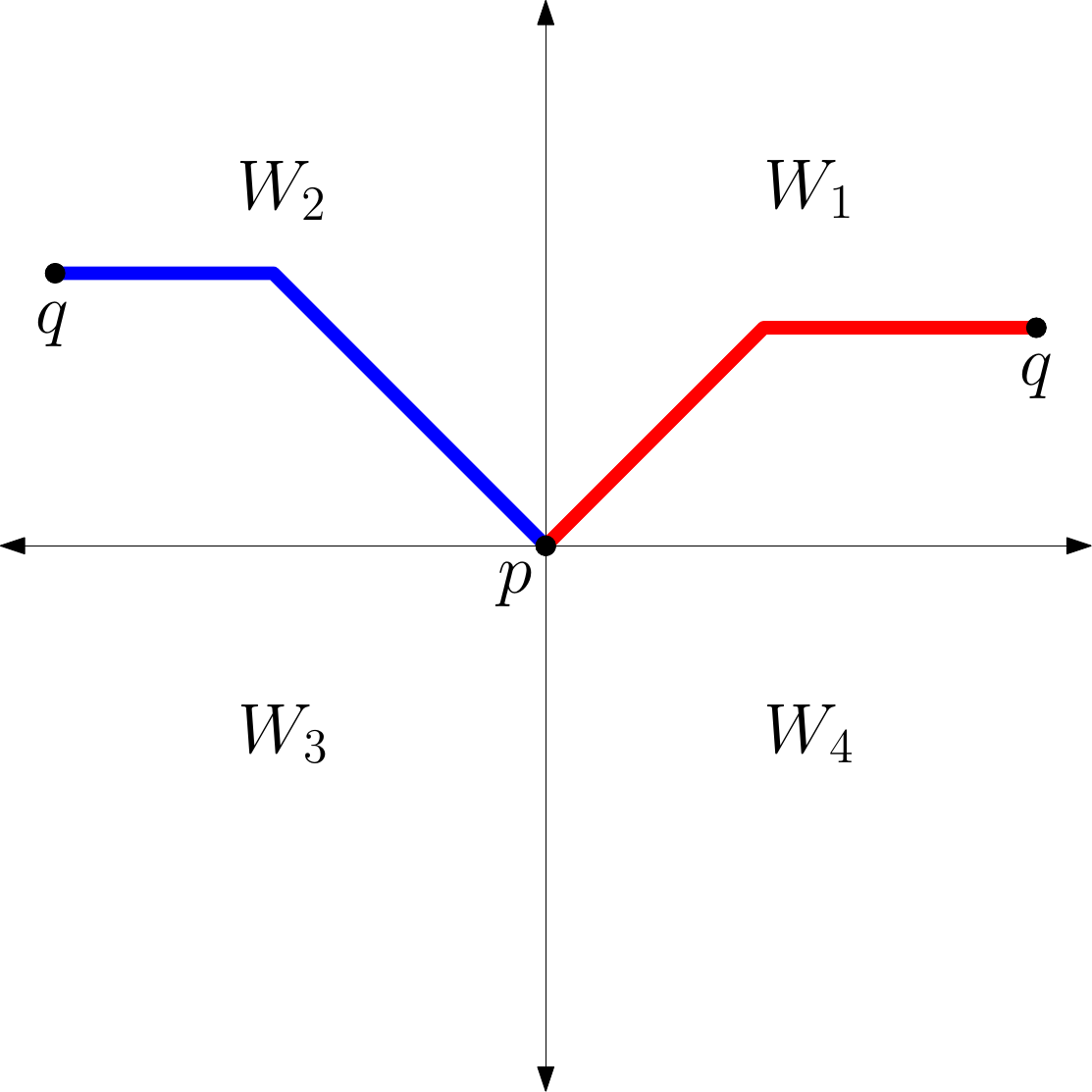}
	\end{center}
	\caption{Left: a degree $3$ spanner on $\Lambda$. Right:  a
          schematic diagram showing the path between $p,q$ when $q$
          lies in different quadrants of $p$ (when $y \leq x$).
          The bold paths consist of segments of lengths $1$ and $\sqrt{2}$. } 
	\label{fig:23}
\end{figure}

\smallskip 
\emph{Case 1}: $q \in W_1$. By the symmetry of $G$, 
we may assume in the analysis that $q=(x,y)$, where $0 \leq y \leq x$. 
Consider the path from $p$ to $q$ via $(y,y)$, that visits every 
lattice point on this diagonal segment as shown in Fig.~\ref{fig:25}
and let $\ell(x,y)$ denote its length. 

\begin{figure}[ht]
	\begin{center}
		\includegraphics[scale=0.42]{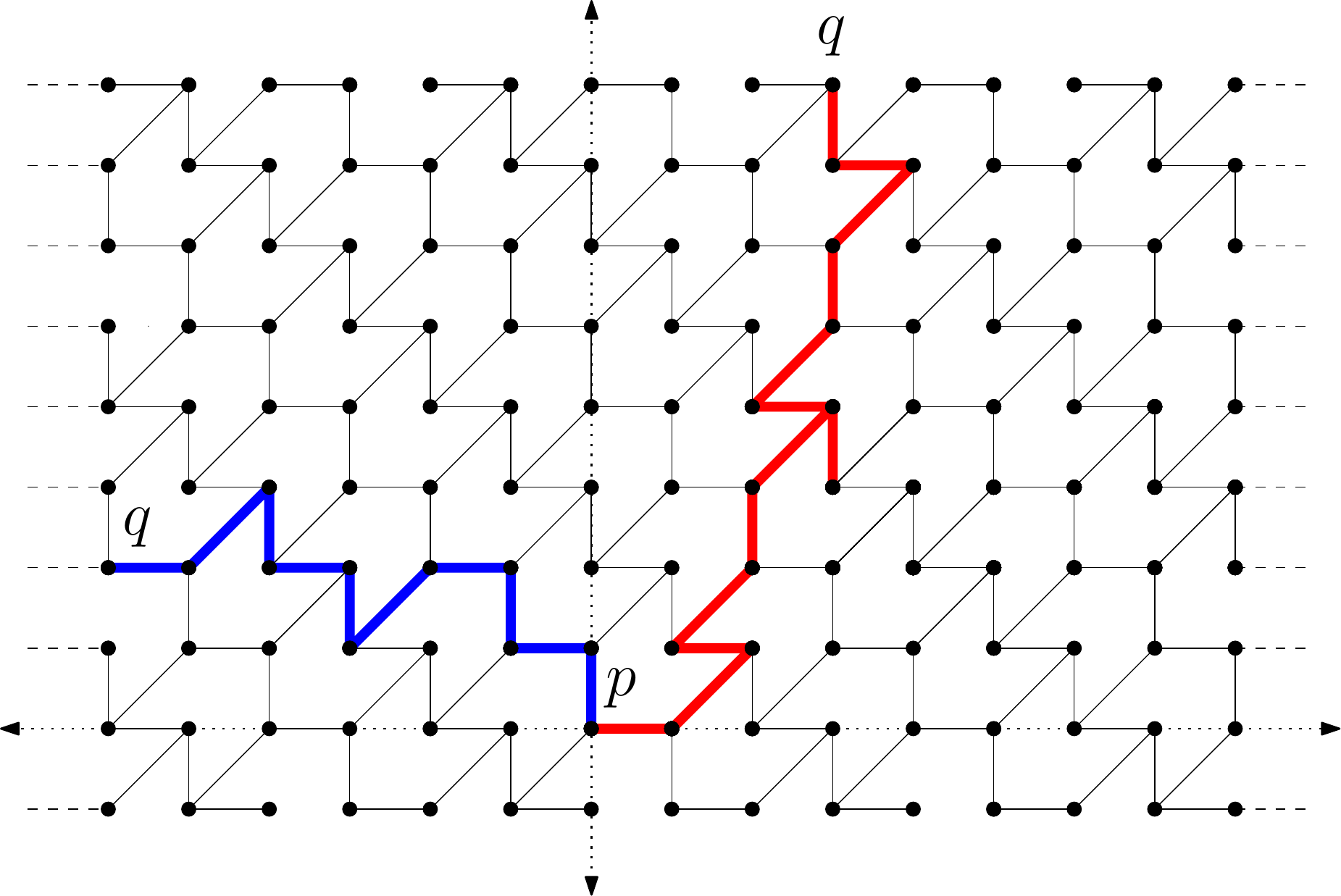}
		\hspace{2mm}
		\includegraphics[scale=0.42]{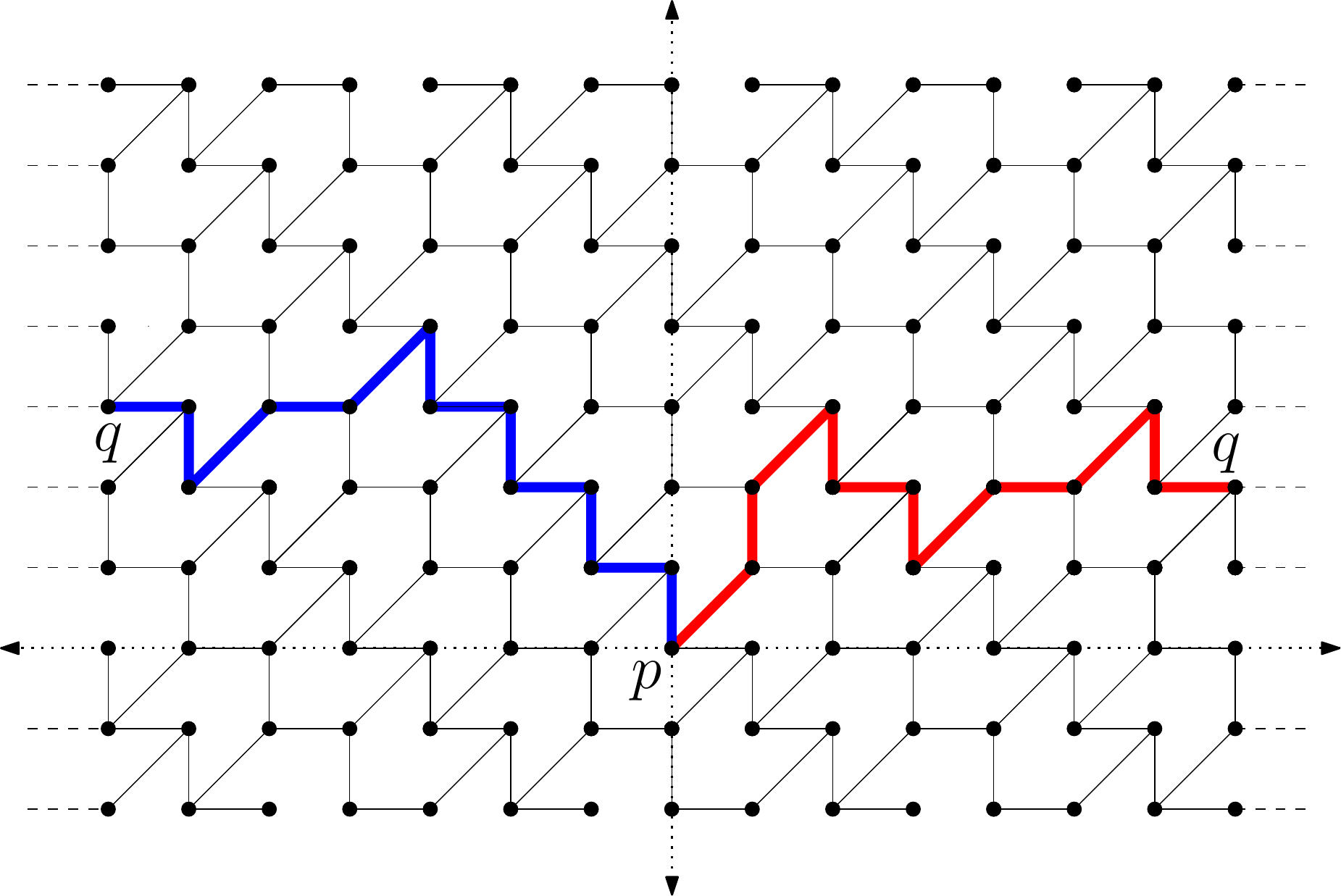}
	\end{center}
	\caption{Illustration of various paths from $p$ to $q$
          depending on the pattern of edges incident to $p$; for $q\in
          W_1$ (in red) and for $q\in W_2$ (in blue).
          Observe that in the red path on the left, 
          a unit vertical edge is traversed in both directions (but can be shortcut).} 
	\label{fig:25}
\end{figure}

If $y=0$, or $x=y$, it is easily checked that the stretch factor is
at most $1+\sqrt2$. Assume subsequently that $x \geq y+1$ and $y \geq
1$. A path of length $(2+2\sqrt{2})\lfloor y/2 \rfloor+(y \bmod
1)(2+\sqrt{2})$ suffices to reach from $p=(0,0)$ to $(y,y)$, and a
path of length $\lceil (x-y)/2 \rceil \sqrt{2} + (x-y)$ suffices to
reach from $(y,y)$ to $q=(x,y)$. Thus, 

$$ \ell(x,y) \leq (2+2\sqrt{2}) \bigg \lfloor \frac{y}{2} \bigg
\rfloor+(y \bmod 1)(2+\sqrt{2}) +  \bigg \lceil \frac{x-y}{2} \bigg
\rceil \sqrt{2} + (x-y). $$ 
That is,
\begin{numcases}{\ell(x,y)\leq }
(2+2\sqrt2) \frac{y}{2} + (x-y) +
\Big \lceil \frac{x-y}{2} \Big \rceil \sqrt{2}, & \text{ for even }$y$ \nonumber \\
(2+2\sqrt2) \frac{y-1}{2} + (2+\sqrt2) + (x-y) +
\Big \lceil \frac{x-y}{2} \Big \rceil \sqrt{2},  & \text{ for odd }$y$. \nonumber
\end{numcases}

The distance $|pq|$ equals $\sqrt{x^2 + y^2}$ in either case,
and so the corresponding stretch factor (bounded in terms of $x,y$) is
\begin{numcases}{\delta(p,q) \leq \gamma(x,y) := }
	\frac{\left(1 + \frac{\sqrt{2}}{2} \right) x +
          \frac{\sqrt{2}}{2} y + \frac{\sqrt{2}}{2}} 
	{\sqrt{x^2 + y^2}}, & \text{ for even }$y$ \label{eq14}\\
 \frac{\left(1 + \frac{\sqrt{2}}{2} \right) x + \frac{\sqrt{2}}{2} y 
 	+ \left( 1+ \frac{\sqrt{2}}{2} \right)}
 {\sqrt{x^2 + y^2}}, & \text{ for odd }$y$. \label{eq15}
\end{numcases}

Consider first the case of even $y$. Since the case $y=0$ has been dealt with, 
we have $y \geq 2$. Setting $\lambda =x/y$ in~\eqref{eq14} and using Fact~\ref{fact1}
in the last step yields 
\begin{equation*}
\delta(p,q) \leq \gamma(x,y) =
\frac{\left(1 + \frac{\sqrt{2}}{2} \right) \lambda +
  \frac{\sqrt{2}}{2}  + \frac{\sqrt{2}}{2y}}{\sqrt{\lambda^2 + 1}} 
\leq \frac{\left(1 + \frac{\sqrt{2}}{2} \right) \lambda + \frac{3\sqrt{2}}{4}}
{\sqrt{\lambda^2 + 1}} \leq \sqrt{\left(1 + \frac{\sqrt{2}}{2} \right)^2 + \frac98} 
< 2.01 < \delta_0.
\end{equation*}

Consider now the case of odd $y$. We have $y \geq 1$ and $x \geq y+1 \geq 2$.
By~\eqref{eq15} we have
$$ \gamma(x,1) =\frac{\left(1 + \frac{\sqrt{2}}{2} \right) x  + ( 1+\sqrt2)}
{\sqrt{x^2 + 1}} :=f(x). $$

We next show that $f$ is decreasing on the interval $[2, \infty)$. Indeed,
$ f'(x)= f_1(x)/(x^2 + 1)^{3/2}$, where
\begin{align*}
f_1(x) &= \left(1 + \frac{\sqrt{2}}{2} \right) (x^2+1)
-x\left[ \left(1 + \frac{\sqrt{2}}{2} \right)x + (1+\sqrt2) \right]\\
&=\left(1+ \frac{\sqrt{2}}{2} \right) -(1+\sqrt{2})x <0, 
\text { for } x \geq 2. 
\end{align*}

Consequently,
$$ \delta(p,q) \leq \gamma(x,1) = f(x) \leq f(2) =  \delta_0, $$
as required. Observe that the above analysis is tight 
for some point pairs with $x=2,y=1$ (that achieve stretch factor $\delta_0$). 

Consider now the remaining case $y \geq 3$. 
Setting $\lambda =x/y$ in (\ref{eq15}) and using Fact~\ref{fact1} in the last step yields
\begin{align*}
\delta(p,q) \leq \gamma(x,y) &\leq
\frac{\left(1 + \frac{\sqrt{2}}{2} \right) \lambda + \frac{\sqrt{2}}{2}  
+ \left( 1 + \frac{\sqrt{2}}{2} \right) \frac{1}{3}}
{\sqrt{\lambda^2 + 1}}
= \frac{\left(1 + \frac{\sqrt{2}}{2} \right) \lambda + \frac{1 + 2\sqrt2}{3}}
{\sqrt{\lambda^2 + 1}} \\
&\leq \sqrt{\left(1 + \frac{\sqrt{2}}{2} \right)^2 + \left( \frac{1 +
    2\sqrt2}{3} \right)^2}   
< 2.14 < \delta_0,
\end{align*}
as required.

\smallskip 
\emph{Case 2}: $q \in W_2$. We may assume that $q=(-x,y)$, where $x \geq y \geq 0$. 
Consider the path from $p$ to $q$ via $(-y,y)$, that visits every 
lattice point on this diagonal segment as shown in Fig.~\ref{fig:25}, 
and let $\ell(x,y)$ denote its length.
The distance $|pq|$ equals $\sqrt{x^2 + y^2}$.

If $y=0$, it is easily checked that the stretch factor is
at most $\sqrt2$, and so we assume subsequently that $y \geq 1$.  
The path length $\ell(x,y)$ is bounded from above as
\begin{align*}
\ell(x,y) &\leq 2y + (x-y) + \Big \lceil \frac{x-y}{2} \Big \rceil \sqrt2
\leq 2y + (x-y) + \frac{x-y+1}{2} \sqrt2 \\
&= \left(1 + \frac{\sqrt{2}}{2} \right) x + 
 \left(1 - \frac{\sqrt{2}}{2} \right) y + \frac{\sqrt{2}}{2}
\leq \left(1 + \frac{\sqrt{2}}{2} \right) x + y.
\end{align*}
Setting $\lambda =x/y$ and using Fact~\ref{fact1} in the last step yields
that the stretch factor is bounded as 
\begin{align*}
\delta(p,q) \leq \gamma(x,y) := 
\frac{\left(1 + \frac{\sqrt{2}}{2} \right)  x + y}{\sqrt{x^2 +y^2}}=
\frac{\left(1 + \frac{\sqrt{2}}{2} \right) \lambda + 1} {\sqrt{\lambda^2 + 1}} 
\leq \sqrt{\left(1 + \frac{\sqrt{2}}{2} \right)^2 +1} < \sqrt4 =2 < \delta_0, 
\end{align*}
as required. 
\smallskip

Next, we determine the degree $4$ dilation of the square lattice.
\begin{theorem}\label{thm:square:4}
  Let $\Lambda$ be the infinite square lattice.
  Then $\delta_0(\Lambda,4) = \sqrt{2}.$ 
\end{theorem}
\begin{proof} Trivially, the (unrestricted degree) dilation of
  four points placed at the four corners of a square is
  $\sqrt{2}$. Thus, $\delta_0(\Lambda) \geq \sqrt{2}$.  
To prove the upper bound, construct a $4$-regular graph $G$ on $\Lambda$ 
by connecting every $(i,j) \in \Lambda$ with its four neighbors
$(i+1,j),(i,j+1),(i-1,j),(i,j-1)$. For any two points $p,q \in \Lambda$,
the Manhattan path connecting them yields a stretch factor of the form 
$$ \frac{x+y}{\sqrt{x^2 + y^2}} \leq \sqrt{2}, \text{ where } x,y \in \NN, $$
as required.
\end{proof}

\paragraph{Remark.} It can be checked that the upper and lower bounds in
Theorem~\ref{thm:square:4} hold for every degree $k \geq 4$.
Thus, $\delta_0(\Lambda,k)=\sqrt{2}$ for $k\geq 4$.

\section{The hexagonal lattice} \label{sec:hexagonal}

This section is devoted to the degree 3 and 4 dilation of the
hexagonal lattice. In~\cite{DG15b}, we  showed that the degree 3
dilation of the infinite hexagonal lattice is between $2$ and
$3$. Here we establish that the exact value is $1+\sqrt{3}$. 
\begin{theorem} \label{thm:hexagonal:3}
Let $\Lambda$ be the infinite hexagonal lattice. Then
$\delta_0(\Lambda,3) = 1+\sqrt{3}.$ 
\end{theorem}
\begin{proof} 
{\textbf{Lower bound.}} Consider a section of the lattice as shown in
Fig.~\ref{fig:f10} (left). 
\begin{figure}[ht]
	\begin{center}
		\includegraphics[scale=0.7]{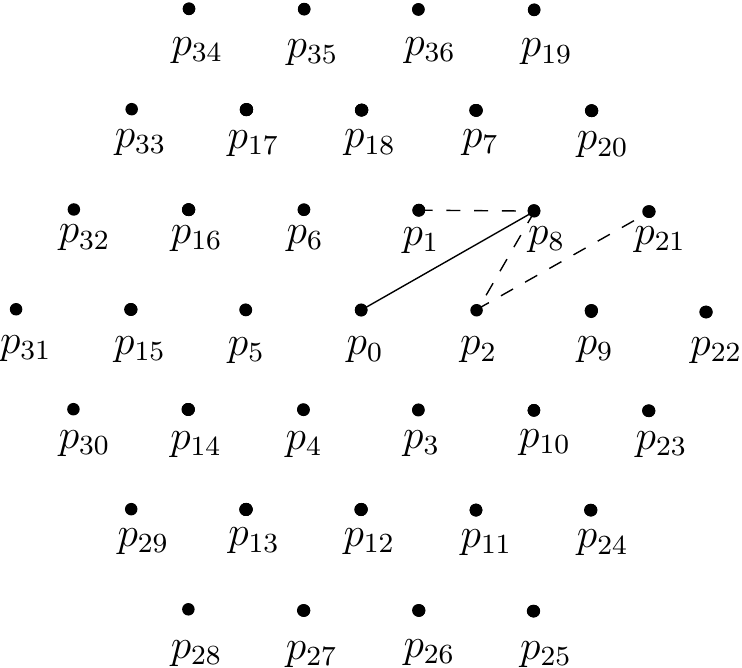}
	\end{center}
	\caption{If an edge of length $\sqrt{3}$ is present, then the
          stretch factor of any plane degree $3$ graph is $\geq
          1+\sqrt{3}$.} 
	\label{fig:f10}
\end{figure} First, we will show that if  an edge of length at least
$\sqrt{3}$ is present, the stretch factor of any resulting plane
degree $3$ graph is at least $1+\sqrt{3}$. Now assume, as we may, that
the edge $p_0p_8$ of length $\sqrt{3}$ is present. Now consider the
point pair $p_1,p_2$. Clearly, $|p_1p_2|=1$. It is easy to check that
between $p_1,p_2$, there are two shortest detours each of length $2$,
viz. $\rho (p_1,p_8,p_2)$ and $\rho (p_1,p_0,p_2)$. The next largest
detours $\rho (p_1,p_8,p_9,p_2)$ and $\rho (p_1,p_0,p_3,p_2)$ have
length $3$ each, in which cases, $\delta(p_1,p_2)\geq 3$. Hence,
without loss of any generality, consider $\rho (p_1,p_8,p_2)$, and
assume that the edges $p_1p_8$ and $p_2p_8$ are present. Then, 
$$\delta(p_8,p_{21}) \geq \frac{|\rho(p_8,p_2,p_{21})|}{|p_8p_{21}|} \geq 1+\sqrt{3}. $$

A similar argument can be made for any edge $e$ of length greater than
$\sqrt{3}$, since one can always locate two lattice points lying in
opposite sides of $e$; as required in the above analysis.  In the
remaining part of the proof, assume that no edge of length
$\sqrt{3}$ or more is present. In particular, we will only consider
unit length edges in our proof. Note that if every point in $\Lambda$
has degree 1 in the graph, we have a matching on $\Lambda$, and hence
the graph is disconnected. Thus, let $p_0$ be any point in $\Lambda$
with degree at least $2$. We have the following two\footnote{As
  mentioned in Section~\ref{sec:intro}, one can argue that degree $3$
  is needed for achieving a constant stretch factor. Thus, it is enough
  to analyze the case when $\deg(p_0)=3$. Nevertheless, we include a
  complete argument.} cases: 
\smallskip

\emph{Case 1:} $\deg(p_0) =2$. There are $3$ non-symmetric sub-cases as follows. 
\begin{figure}[ht]
\begin{center}
\includegraphics[scale=0.7]{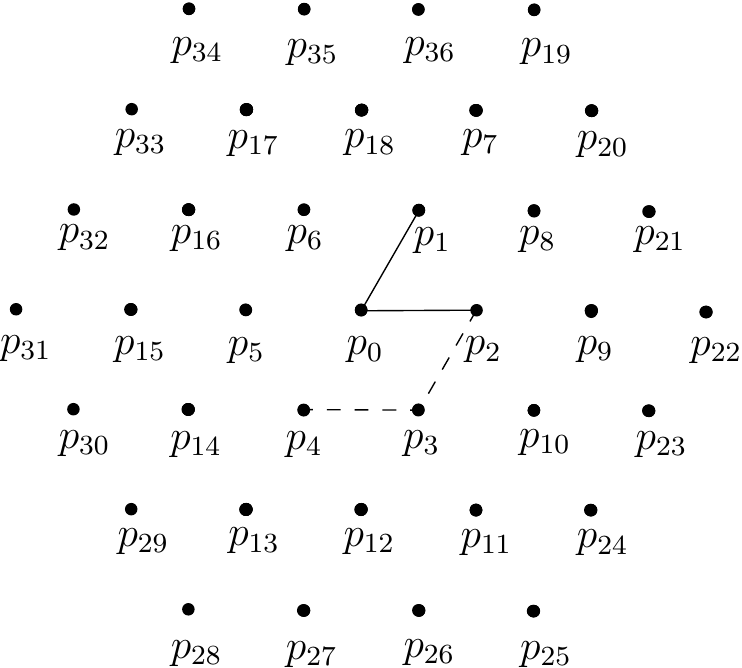}\hspace{2mm}
\includegraphics[scale=0.7]{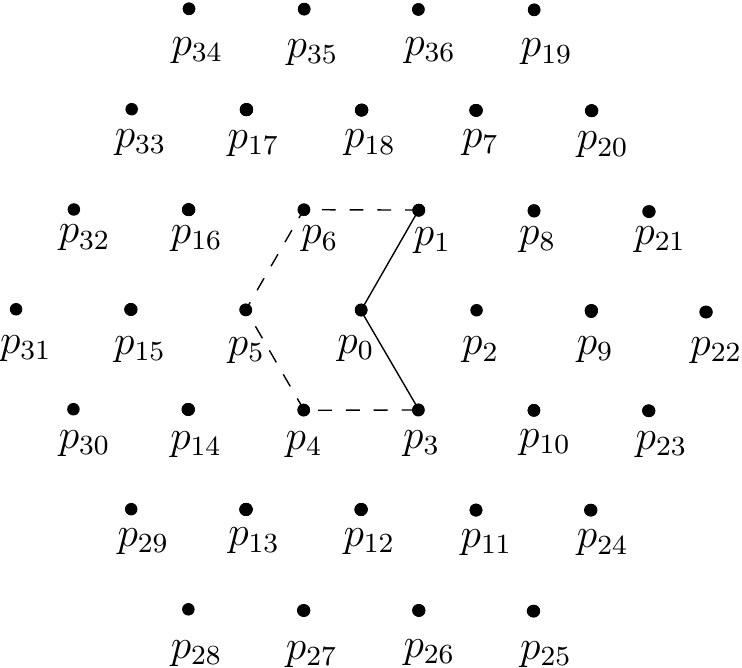}\hspace{2mm}
\includegraphics[scale=0.7]{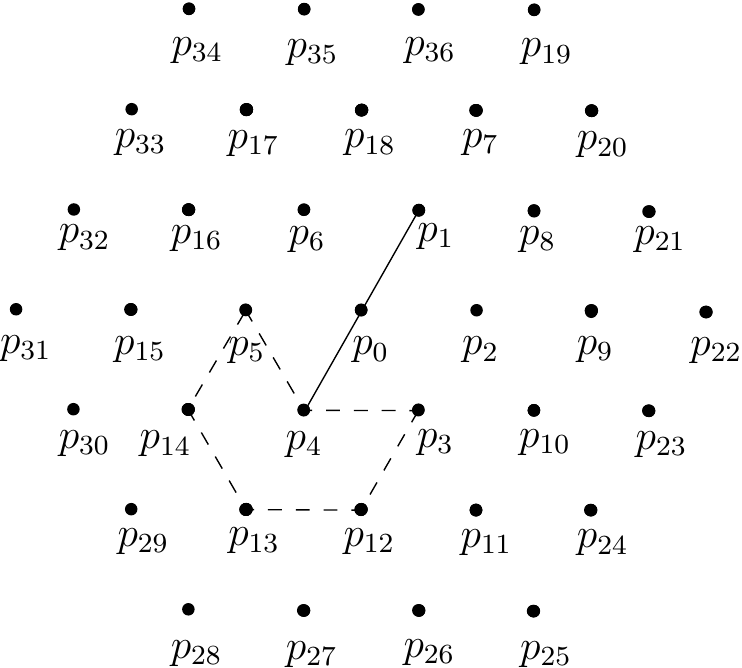}
\end{center}
\caption{Illustration of \emph{Case 1} from the proof of lower bound
  in Theorem~\ref{thm:hexagonal:3}. Left: \emph{Case 1.1}, Middle:
  \emph{Case 1.2}, Right: \emph{Case 1.3}.} 
\label{fig:f11}
\end{figure}

\smallskip
\emph{Case 1.1:} Refer to Fig.~\ref{fig:f11} (left). Let the edges
$p_0p_1$, $p_0p_2$ be present. Then,  
$$ \delta(p_0,p_4) \geq \frac{|\rho(p_0,p_2,p_3,p_4)|}{|p_0p_4|} \geq 3. $$

\emph{Case 1.2:} Refer to Fig.~\ref{fig:f11} (middle). Now, let the
edges $p_0p_1$, $p_0p_3$ be present. Then, 
$$ \delta(p_0,p_5) \geq \frac{|\rho(p_0,p_3,p_4,p_5)|}{|p_0p_5|} =
\frac{|\rho(p_0,p_1,p_6,p_5)|}{|p_0p_5|} \geq 3. $$ 

\emph{Case 1.3:} Refer to Fig.~\ref{fig:f11} (right). Let the edges
$p_0p_1$, $p_0p_4$ be present. Note that if the edge $p_3p_4$ is
absent, $\delta(p_0,p_3) \geq 3$. So, assume that $p_3p_4$ is
present. Similarly let $p_4p_5$ be present otherwise $\delta(p_0,p_5)
\geq 3$. Then, arguing the same way as in \emph{Case 1.2},
$\delta(p_4,p_{13}) \geq 3.$

\medskip
\emph{Case 2:} $\deg(p_0) =3$.  There are  $3$ non-symmetric sub-cases as follows. 
\begin{figure}[ht]
	\begin{center}
		\includegraphics[scale=0.7]{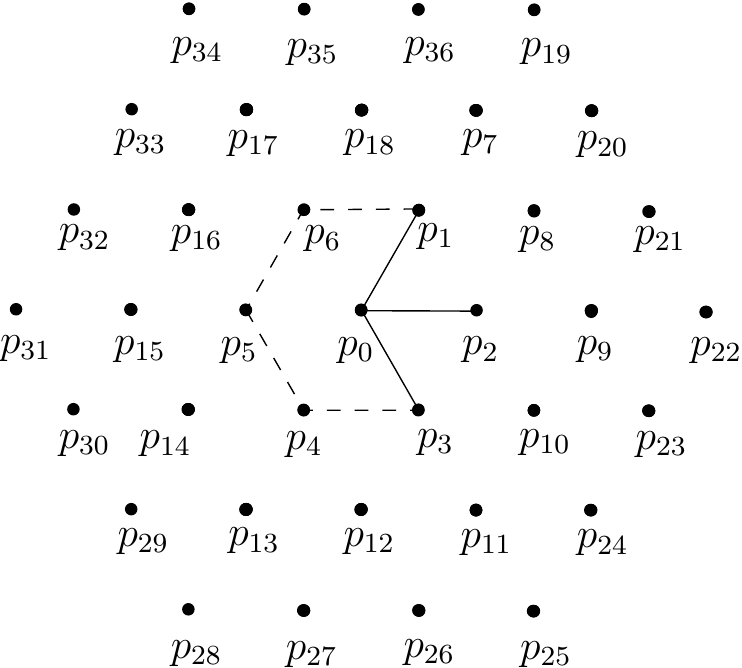}\hspace{2mm}
		\includegraphics[scale=0.7]{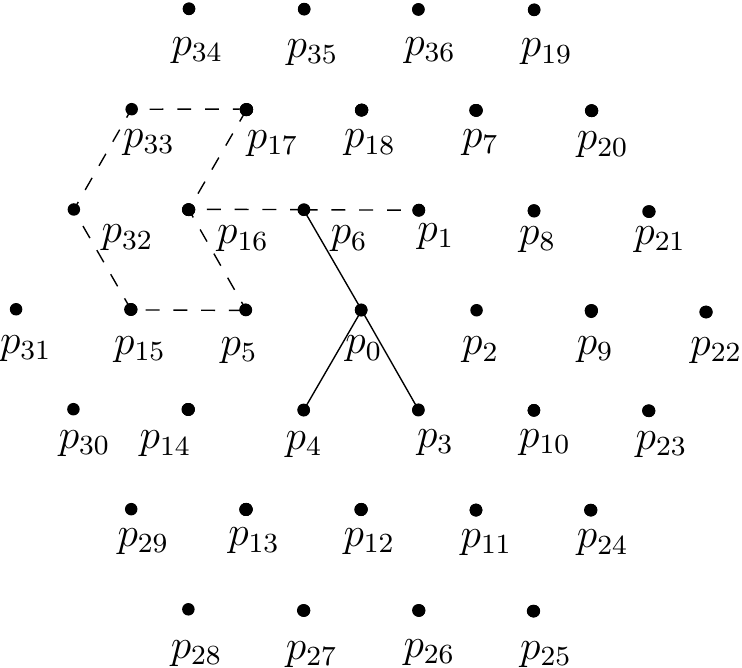}\hspace{2mm}
		\includegraphics[scale=0.7]{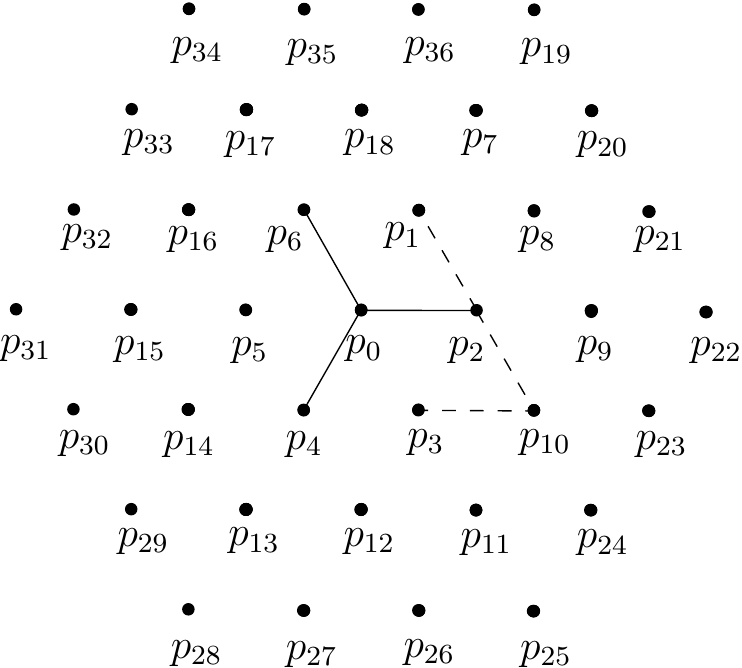}
	\end{center}
	\caption{Illustration of \emph{Case 2} from the proof of lower
          bound in Theorem~\ref{thm:hexagonal:3}. Left: \emph{Case
            2.1}, Middle: \emph{Case 2.2}, Right: \emph{Case 2.3}.} 
	\label{fig:f14} 
\end{figure}

\emph{Case 2.1:} Refer to Fig.~\ref{fig:f14} (left). Let the edges
 $p_0p_1$, $p_0p_2$, $p_0p_3$ be present. Then, by a similar argument
 as in \emph{Case 1.2}, $\delta(p_0,p_5) \geq 3$. 

\smallskip
\emph{Case 2.2:} Refer to Fig.~\ref{fig:f14} (middle). Now, let the
edges $p_0p_3$, $p_0p_4$, $p_0p_6$ be present. Clearly, if  $p_1p_6$
is absent, $\delta(p_0,p_1) \geq 3$. Thus, assume that $p_1p_6$ is
present. Now consider the pair $p_5,p_6$. If $p_5p_6$ is present, then
$\delta(p_6,p_{17}) \geq 3$, arguing in a similar way to \emph{Case
  1.2}. Thus, assume that $p_5p_6$ is absent. The shortest detour
between $p_5,p_6$ is $\rho(p_5,p_{16},p_6)$ which has length $2$. The
next largest detour has length $3$. So, let the edges $p_6p_{16}$ and
$p_5p_{16}$ be present. Now consider the pair $p_{16},p_{17}$. If
$p_{16}p_{17}$ is present, $\delta(p_{16},p_{32}) \geq 3$ (analysis is
similar to \emph{Case 1.2}), otherwise, $\delta(p_6,p_{17}) \geq 3$.

\smallskip
\emph{Case 2.3:} Refer to Fig.~\ref{fig:f14} (right). Let  $p_0p_2$,
$p_0p_4$, $p_0p_6$ be present. To achieve $\delta(p_0,p_1)=2$, at
least one of $p_1p_2$ or $p_1p_6$ needs to be present (the next
largest detour has length $3$). Without loss of any generality, assume
that $p_1p_2$ is present. Now consider the pair $p_2,p_3$. If $p_2p_3$
is present, then by \emph{Case 2.1}, $\delta(p_2,p_9) \geq 3$. So,
assume that $p_2p_3$ is absent. The minimum length detour is
$\rho(p_2,p_{10},p_3)$ (next largest detours have length $3$
each). Thus, let $p_2p_{10}$ and $p_3p_{10}$ be present. Now, observe
that the edges incident to $p_2$ form the same symmetric pattern as
dealt with in \emph{Case 2.2}, where it is shown that the stretch
factor of any resulting degree $3$ plane graph is at least $3$. 
\begin{figure}[ht]
\begin{center}
\includegraphics[scale=0.45]{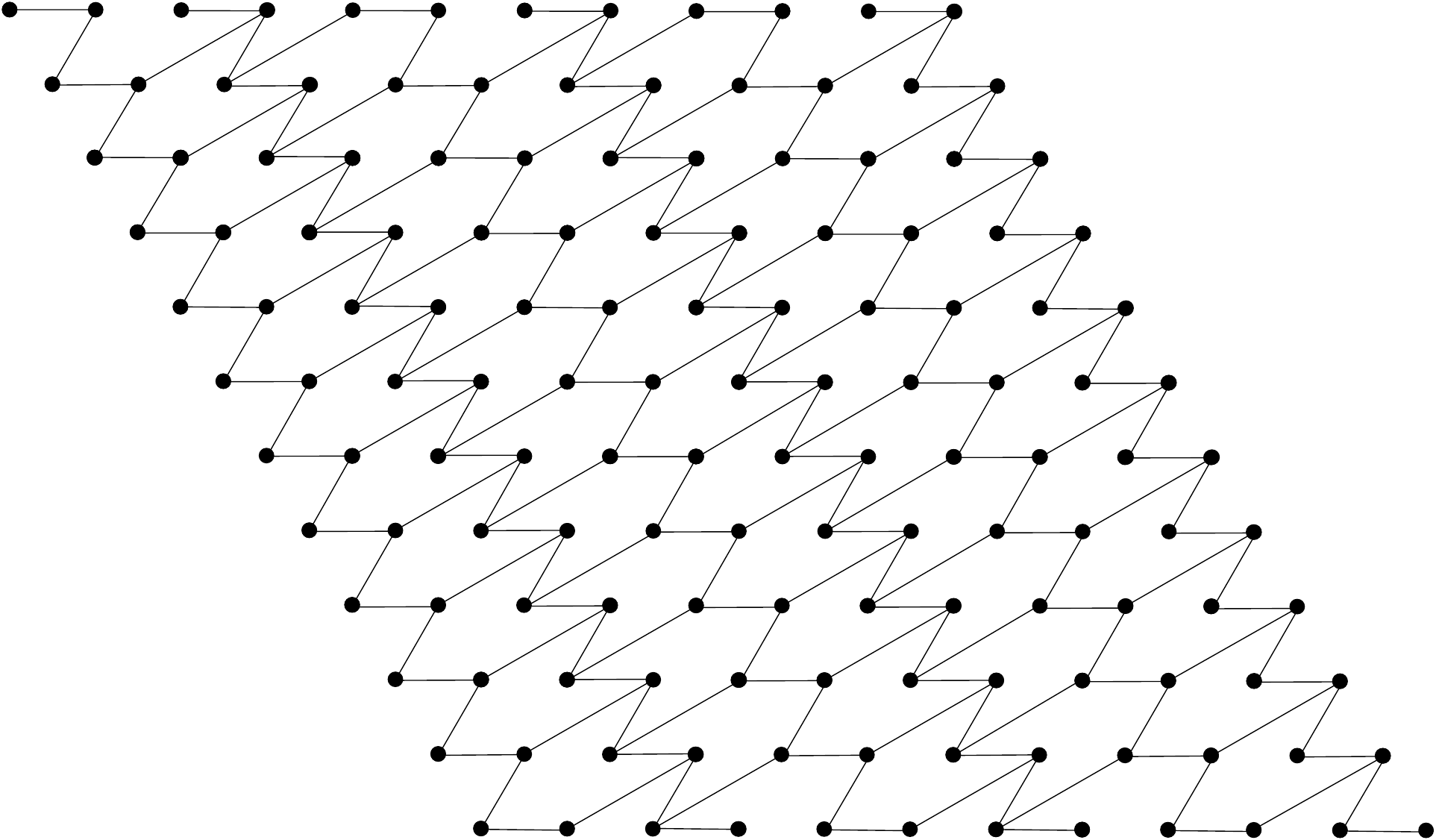}
\includegraphics[scale=0.37]{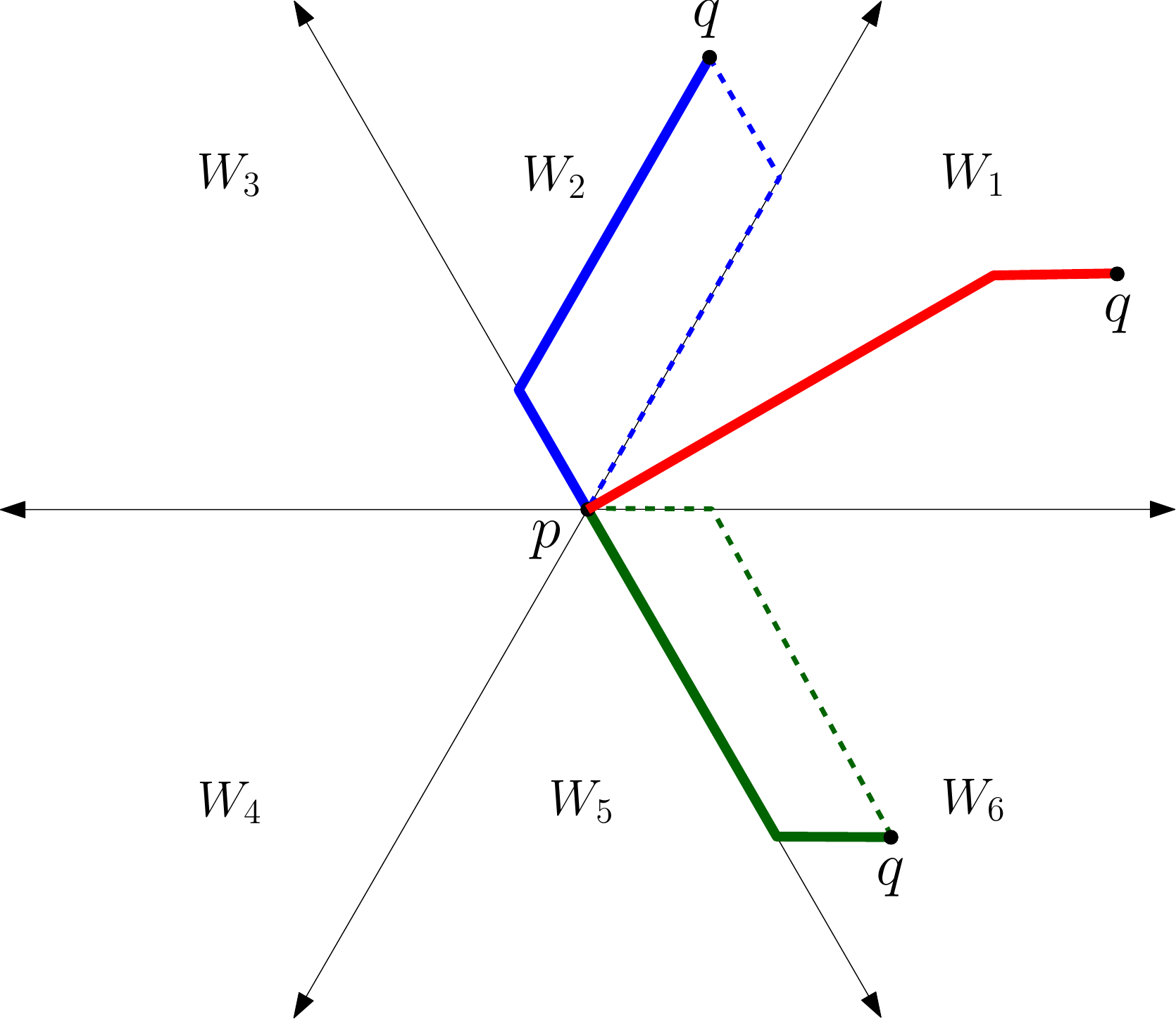}
\end{center}
\caption{Left: a degree $3$ plane graph on $\Lambda$. Right: a
schematic diagram showing the path between $p$ and $q$ when $q$
lies in different wedges determined by $p$. The bold paths
consist of segments of lengths $1$ and $\sqrt{3}$. 
Alternative paths are shown using dotted segments.}	 
\label{fig:f18}
\end{figure}

\smallskip
\textbf{{Upper bound.}} 
We construct a $3$-regular graph $G$ achieving 
$\delta_0(\Lambda,3) \leq 1+\sqrt{3}$, 
as illustrated in Fig.~\ref{fig:f18} (left). 
For any two lattice points $p,q \in \Lambda$, we construct a path in $G$.
Set $p$ as the origin, and 
subdivide the plane into six wedges of $60^\circ$ each, centered at
$p$, and labeled counterclockwise $W_i$, $i=1,\ldots,6$, as in
Fig.~\ref{fig:f18} (right). 
Points on the dividing lines are assigned arbitrarily to any of the two
adjacent wedges. Let $\theta =\pi/3$, and consider the three unit vectors
$\vec{\mu_i}= (\cos {i \theta}, \sin{i \theta})$, for $i=0,1,2$.
We distinguish three cases depending on the location of $q$.
	
\smallskip 
\emph{Case 1}: $q \in W_1$ (the case $q \in W_4$ is symmetric), \ie, 
$\vec{q}= u \vec{\mu_0} + v \vec{\mu_1}$, for some $u,v \in \NN$. 
By the symmetry of $G$, we can assume that $u \geq v \geq 0$ in the analysis. 
Consider the path from $p$ to $q$ via $v \vec{\mu_0} + v \vec{\mu_1}$ 
that visits every lattice point on the diagonal segment,  
as shown in Fig.~\ref{fig:f20} and let $\ell(u,v)$ denote
its length. Observe that connecting 
$a\vec{\mu_0}+a\vec{\mu_1}$ to $(a+2)\vec{\mu_0}+(a+2)\vec{\mu_1}$ 
requires a length of $2+2\sqrt{3}$. Thus, 
$$ \ell(u,v) \leq (2+2\sqrt{3}) \bigg \lfloor \frac{v}{2} \bigg \rfloor + 
(2+\sqrt{3})(v\bmod 1) + (u-v) + \bigg \lceil \frac{u-v}{2} \bigg \rceil \sqrt{3}. $$

\begin{figure}[ht]
\begin{center}		
	\includegraphics[scale=0.39]{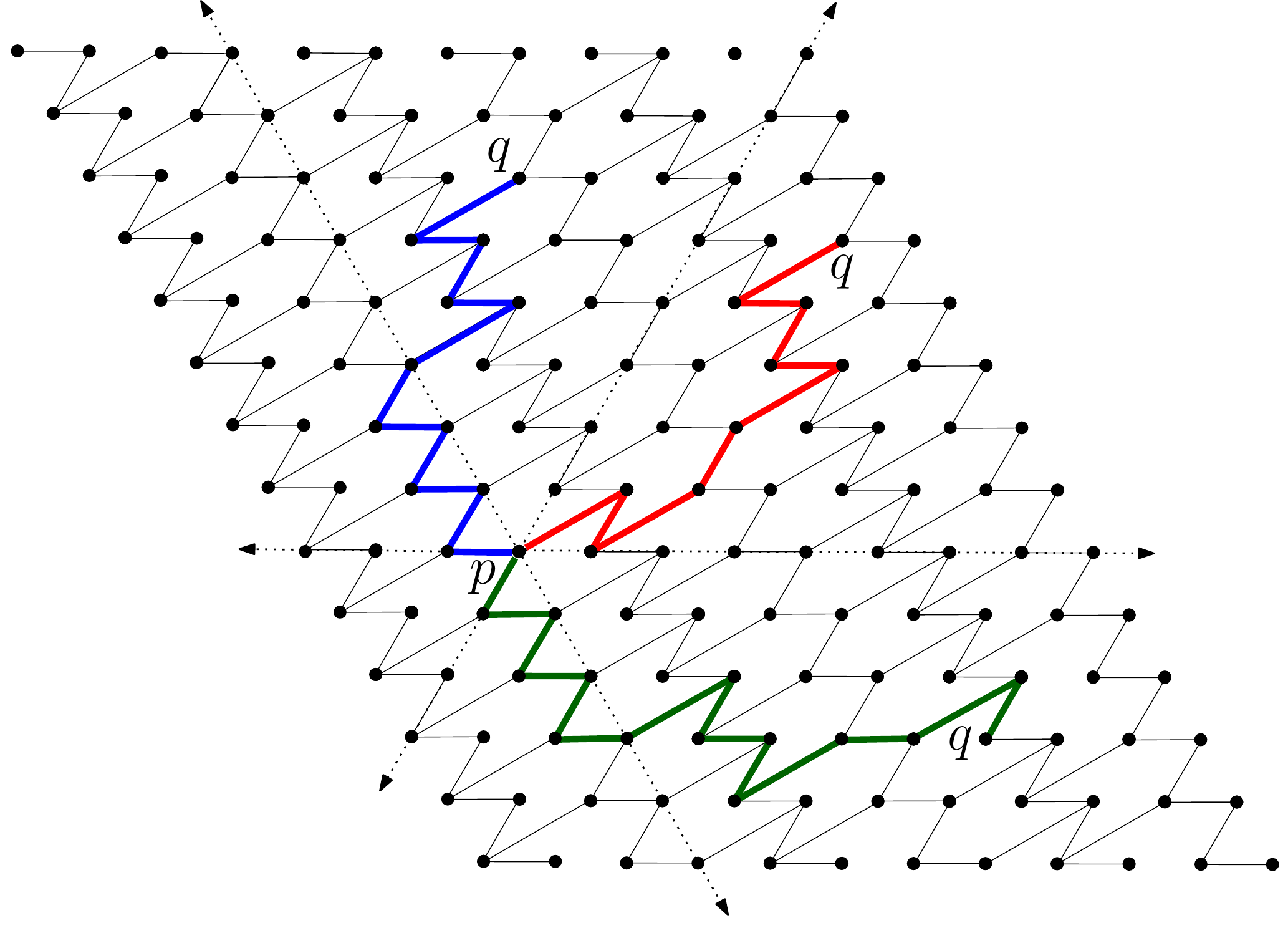}
	\includegraphics[scale=0.39]{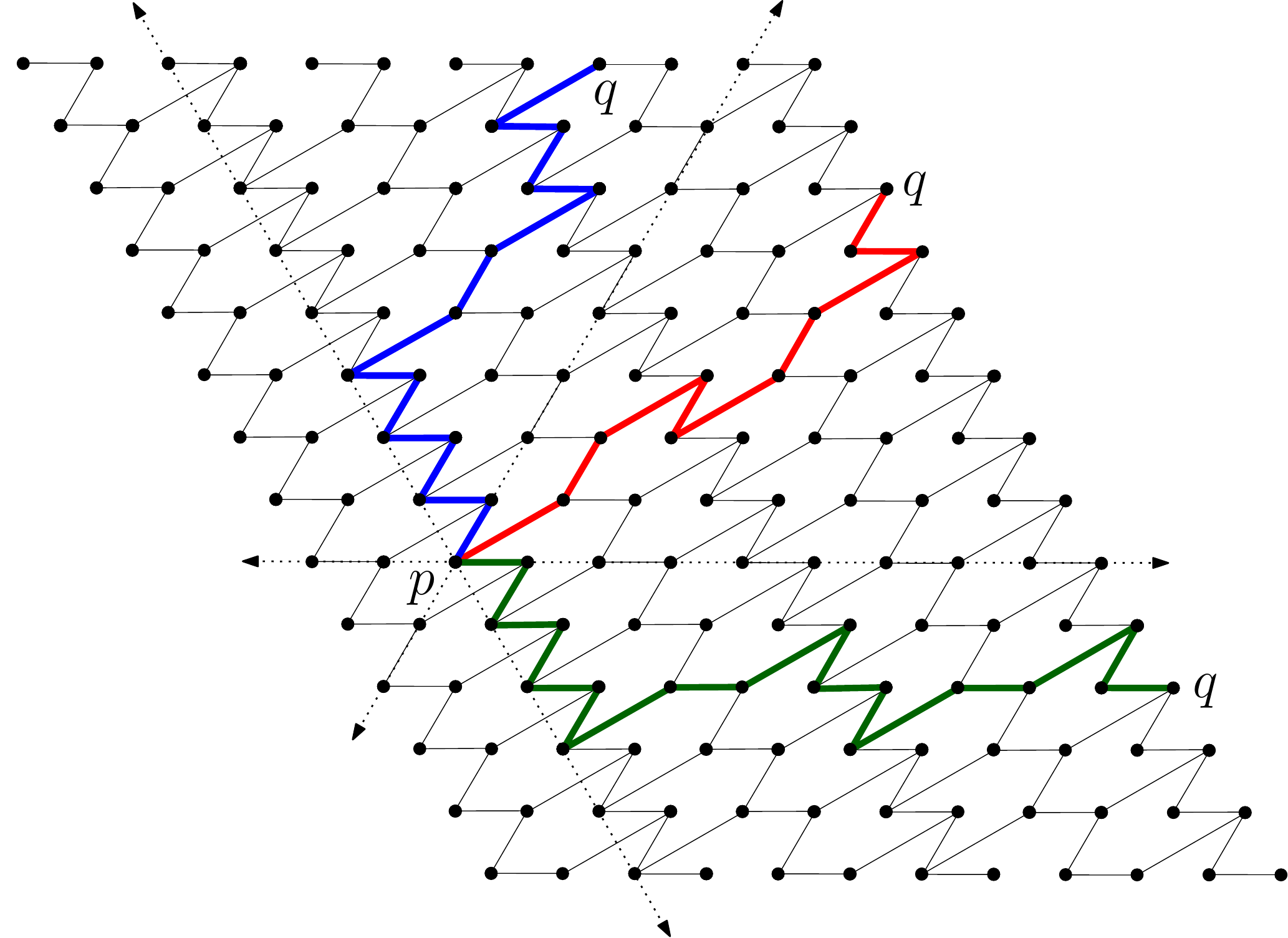}
\end{center}
\caption{Illustration of various paths from $p$ to $q$ depending on
the pattern of edges incident to $p$; for $q \in W_1$ (in red), 
for $q \in W_2$ (in blue), and for $q \in W_6$ (in green).}	
\label{fig:f20}
\end{figure}
	
For even $v$, $\ell(u,v)$ is bounded from above by
\begin{equation*} 
	\ell(u,v) \leq \left(2 + 2 \sqrt{3}\right) \frac{v}{2} 
	+ \left(1 + \frac{\sqrt{3}}{2} \right) (u-v) 
	+ \frac{\sqrt{3}}{2} 
	= \left(1 + \frac{\sqrt{3}}{2} \right) u +
        \frac{\sqrt{3}}{2} v + \frac{\sqrt{3}}{2}. 
\end{equation*}
		
For odd $v$, $\ell(u,v)$ is bounded from above by
\begin{align*} 
\ell(u,v) &\leq \left(2 + 2 \sqrt{3}\right) \frac{v-1}{2} + (2 + \sqrt{3}) +
\left(1 + \frac{\sqrt{3}}{2} \right) (u-v) + \frac{\sqrt{3}}{2} \nonumber \\
&= \left(1 + \frac{\sqrt{3}}{2} \right) u + \frac{\sqrt{3}}{2} v 
+ \left(1 + \frac{\sqrt{3}}{2} \right). 
\end{align*}

The distance $|pq|$ equals
$$ \sqrt{u^2 + v^2 - 2uv \cos \frac{2\pi}{3}} = \sqrt{u^2 + v^2 + uv}, $$
and so the corresponding stretch factor $\delta(p,q)$ is bounded by a
function $\gamma(u,v)$ as follows
\begin{numcases}{\delta(p,q) \leq \gamma(u,v) :=}
\frac{\left(1 + \frac{\sqrt{3}}{2} \right) u + \frac{\sqrt{3}}{2} v + \frac{\sqrt{3}}{2}} 
{\sqrt{u^2 + v^2 + uv}}, & \text{ for even } $v$ \nonumber \\
\frac{\left(1 + \frac{\sqrt{3}}{2} \right) u + \frac{\sqrt{3}}{2} v 
+ \left(1 + \frac{\sqrt{3}}{2} \right)}
{\sqrt{u^2 + v^2 + uv}}, & \text{ for odd } $v$. \nonumber 
\end{numcases}

Consider first the case of even $v$. We have $u \geq v \geq 0$ and
$u \geq 1$ (since $u=v=0$ is not a valid choice). 
We next show that $\gamma(u,v)$ is a decreasing function of $v$ for $v \geq 0$. 
Indeed, 
$$ \frac{\partial \gamma(u,v)}{\partial v} = f(u,v)/[2(u^2 + v^2 + uv)^{3/2}], $$
where
\begin{align*}
f(u,v) &= \sqrt{3} \left(u^2 + v^2 + uv\right) -
\left(2v+u\right) \left[\left(1 + \frac{\sqrt{3}}{2} \right) u +
\frac{\sqrt{3}}{2} v + \frac{\sqrt{3}}{2} \right] \nonumber \\
&= \left(\frac{\sqrt{3}}{2} -1 \right) u^2 
- \left(2 + \frac{\sqrt{3}}{2} \right) uv
- \frac{\sqrt{3}}{2} u -\sqrt{3} v <0.
\label{eq3}
\end{align*}
Consequently,
$$ \delta(p,q) \leq \gamma(u,v) \leq \gamma(u,0) = 
\frac{\left(1 + \frac{\sqrt{3}}{2} \right) u + \frac{\sqrt{3}}{2}}{u}
= 1 + \frac{\sqrt{3}}{2} + \frac{\sqrt{3}}{2u} \leq 1 + \sqrt{3}, $$
as required.
	
Consider now the case of odd $v$. We have $u \geq v \geq 1$. 
Since the expressions of $\gamma(u,v)$ for odd and even $v$ differ by
$(u^2 + v^2 + uv)^{-1/2}$, which is also a decreasing function of $v$, 
it follows that 
$\gamma(u,v)$ for odd $v$ is decreasing on the same interval, in particular on the
interval $v \geq 1$. Consequently,
$$ \delta(p,q) \leq \gamma(u,v) \leq \gamma(u,1)= 
\frac{\left(1 + \frac{\sqrt{3}}{2} \right) u + (1+\sqrt{3})}
{\sqrt{u^2 +u +1}} \leq \frac32 + \frac{2}{\sqrt{3}} < 1 + \sqrt{3}, $$
as required. To check this last inequality, let
$$ h(u)= \frac{\left(1 + \frac{\sqrt{3}}{2} \right) u + (1+\sqrt{3})}
{\sqrt{u^2 +u +1}}, $$ 
and notice that this function is decreasing for $u \geq 1$, thus
$h(u) \leq h(1)= \frac32 + \frac{2}{\sqrt{3}}$.

\smallskip 
\emph{Case 2}: $q \in W_2$ (the case $q \in W_5$ is symmetric), \ie, 
$\vec{q}= u \vec{\mu_2} + v \vec{\mu_1}$, for some $u,v \in \NN$. 
Consider the path from $p$ to $q$ via $u \vec{\mu_2}$
as shown in Fig.~\ref{fig:f20}
and let $\ell(u,v)$ denote its length. (Alternatively, the
path via $v\vec{\mu_1}$ can be used.) 
Then $\ell(u,v)$ is bounded from above as follows
$$ \ell(u,v) \leq 2u + v+ \bigg \lceil \frac{v}{2} \bigg \rceil \sqrt{3} 
\leq 2u + \left( 1 + \frac{\sqrt3}{2} \right) v + \frac{\sqrt3}{2}. $$ 

As in Case 1, the distance $|pq|$ equals
$ \sqrt{u^2 + v^2 - 2uv \cos \frac{2\pi}{3}} = \sqrt{u^2 + v^2 + uv}$,
and so the corresponding stretch factor is
$$ \delta(p,q) \leq \gamma(u,v) := 
\frac{2u + \left(1 + \frac{\sqrt{3}}{2} \right) v + \frac{\sqrt{3}}{2}}
{\sqrt{u^2 + v^2 + uv}}. $$
	
We can assume that $u \geq 1$ and $v \geq 1$ (else the stretch factor
is at most $1 +\sqrt3$). 
Further, since the coefficient of $u$ is larger than that of $v$ in the numerator, 
we can assume that $u \geq v \geq 1$ when maximizing $\gamma(u,v)$.
Set now $\lambda =\frac{u}{v} \geq 1$. 
We have 
$$ \gamma(u,v) = 
\frac{2 \lambda + \left( 1 + \frac{\sqrt3}{2} \right) + \frac{\sqrt3}{2v}}
{\sqrt{\lambda^2 + \lambda +1}}
\leq \frac{2 \lambda + (1 + \sqrt3)} {\sqrt{\lambda^2 + \lambda +1}} := f(\lambda).
$$
It is easy to check that $f(\lambda)$ is decreasing for 
$\lambda \geq 1$, hence for $u,v \geq 1$ we also have 
$$ \delta(p,q) \leq \gamma(u,v) \leq f(1) = 1+\sqrt{3}, $$
as required.

\smallskip 
\emph{Case 3}: $q \in W_6$ (the case $q \in W_3$ is symmetric), \ie, 
$\vec{q}= -u \vec{\mu_2} + v \vec{\mu_0}$, for some $u,v \in \NN$. 
By the symmetry of $G$, this case is symmetric to Case 2.
	
\medskip
This completes the case analysis and thereby the proof of the upper bound.
\end{proof}

\paragraph{Remark.} In~\cite{DG15a} we have shown that a certain $13$-point section
of the hexagonal lattice with six boundary points removed has degree $3$ dilation
at least $1 +\sqrt3$. It is worth noting that
this subset cannot be used however to deduce that the degree $3$ dilation of the
hexagonal lattice is at least $1 +\sqrt3$. Indeed, the reason is that the absence
of the respective boundary points has been explicitly invoked in that argument. 
This is the reason of why in the proof of the lower bound in Theorem~\ref{thm:hexagonal:3}
we have used a different argument. 

\medskip
Next we determine the degree $4$ dilation of the infinite hexagonal lattice.  
\begin{theorem} \label{thm:hexagonal:4}
	Let $\Lambda$ be the infinite hexagonal lattice. 
Then $ \delta_0(\Lambda,4) = 2$.
\end{theorem}
\begin{proof}
We first prove the lower bound.
  Let $p_0$ be any point in $\Lambda$ with its six closest neighbors,
  say, $p_1,\ldots,p_6$, where $|p_0 p_i|=1$, for $i=1,\ldots,6$. 
  Since $\deg(p_0) \leq 4$ in any plane degree 4 geometric spanner on $\Lambda$,
  there exists $i \in \{1,\ldots,6\}$ such that the edge $p_0p_i$ is absent;
  we may assume that $i=1$. Then
\begin{equation*}
  \delta(p_0,p_1) \geq \frac{|\rho(p_0,p_i,p_1)|}{|p_0p_1|} \geq 2,
  \text{ where } i \in \{2,6\}.
\end{equation*} 

To prove the upper bound, consider the $4$-regular graph $G$ shown in
Fig.~\ref{fig:hex2}; it remains to show that $ \delta(G) \leq 2$.
For any two lattice points $p,q \in \Lambda$, we construct a path in
$G$. Consider the setup from the proof of Theorem~\ref{thm:hexagonal:3}.
Set the lower point $p$ as the origin $(0,0)$.
\begin{figure}[hbtp]
	\centering
	\includegraphics[scale=0.5]{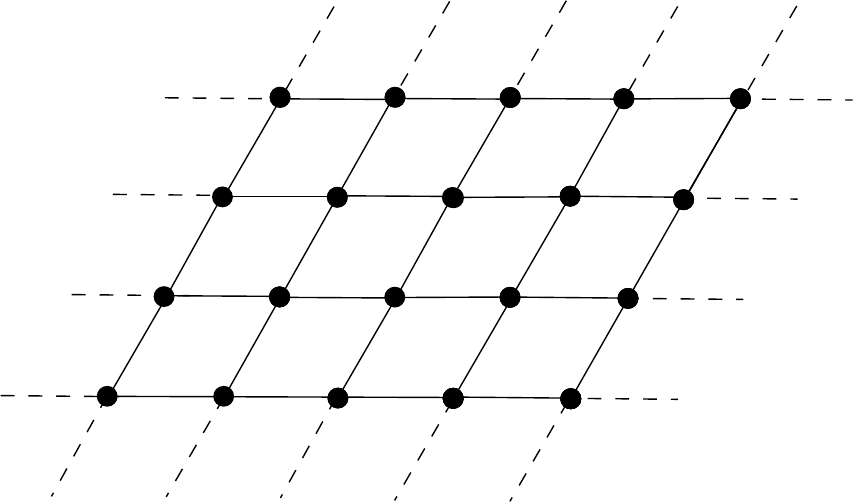}
	\includegraphics[scale=0.5]{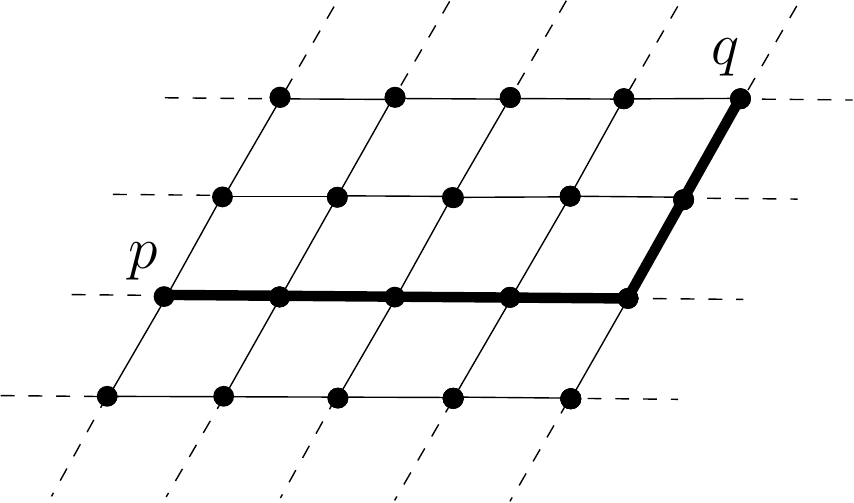}
	\includegraphics[scale=0.5]{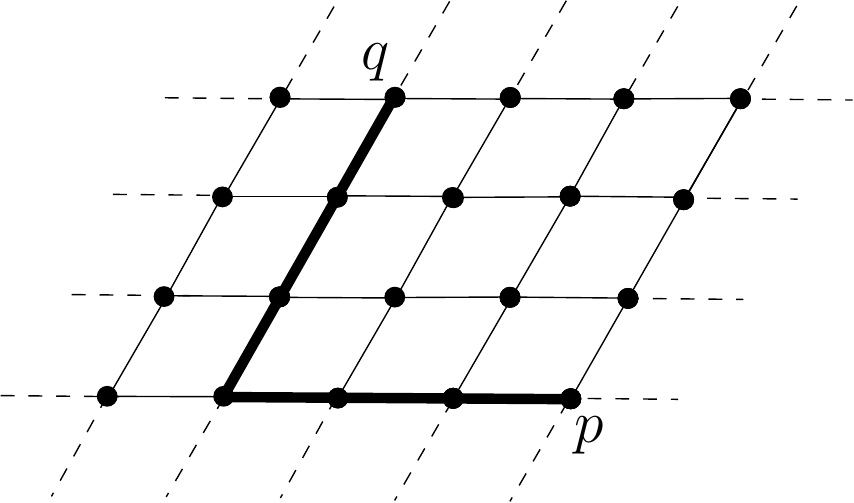}
	\caption{Left: a degree $4$ plane graph $G$ on
        $\Lambda$. Middle, Right: illustration of various paths from $p$ to $q$ 
        depending on their relative position in $\Lambda$.} 
	\label{fig:hex2}
\end{figure}
Let $\theta =\pi/3$, and consider the two unit vectors
$\vec{\mu_i}= (\cos {i \theta}, \sin{i \theta})$, for $i=0,1$.
Then $\vec{q}= \pm u \vec{\mu_0} + v \vec{\mu_1}$, for some $u,v \in \NN$.
Since the two points can be connected by a path in $G$ of length $u+v$,
and the distance between the points is $\sqrt{u^2 + v^2 \pm uv}$
(depending on their relative position in $\Lambda$),
the corresponding stretch factor satisfies
\begin{equation} \label{eq:1}
\delta(p,q) \leq \gamma(u,v) :=\frac{u+v}{\sqrt{u^2 + v^2 \pm uv}} \leq 2,
\end{equation}
as required. Indeed, the above inequalities are
equivalent to $(u \pm v)^2 \geq 0$, which are obvious.
\end{proof}

\paragraph{Remarks.} 

1. Another degree $4$ spanner for the hexagonal lattice
with stretch factor $2$ appears in Fig.~\ref{fig:deg4spanner};
the proof for the stretch factor is left to the reader. 
\begin{figure}[htbp]
\centering
  \includegraphics[scale=0.55]{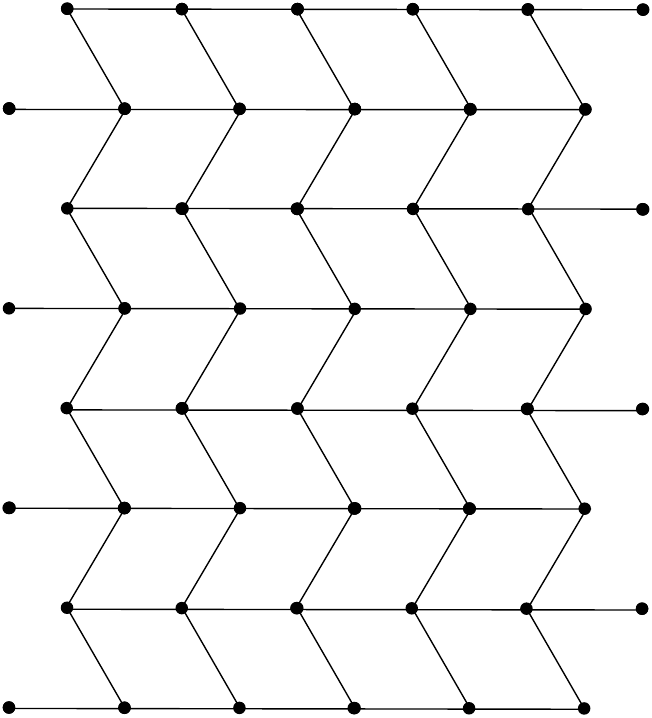}
  \caption{A degree $4$ spanner on $\Lambda$ with stretch factor $2$.}
  \label{fig:deg4spanner}
\end{figure}

\smallskip
2. Note that $\delta_0(\Lambda,5) = 2$ since the bounds in
Theorem~\ref{thm:hexagonal:4} also hold for degree~$5$. 
Let now $k=6$. Connecting each lattice point with all its six closest neighbors
yields a planar graph with stretch factor $2/\sqrt3$. Indeed, as in~\eqref{eq:1},
the stretch factor satisfies
$$ \delta(p,q) \leq \gamma(u,v) :=\frac{u+v}{\sqrt{u^2 + v^2 + uv}} \leq \frac{2}{\sqrt3}, $$
where the above inequality is equivalent to $(u-v)^2 \geq 0$, which is obvious.
Hence $\delta_0(\Lambda,6) \leq 2/\sqrt3$. On the other hand,
an argument similar to that in the proof of the inequality
$\delta_0(\Lambda,3) \geq 1+\sqrt3$ shows that the presence of any edge longer
than $1$ would force the stretch factor to be at least $2$.
We may thus assume that the spanner $G$ contains all unit edges (since no two cross each other);
now the length of a shortest path in $G$ connecting any pair of lattice points at distance
$\sqrt3$ is $2$, thus the stretch factor $2/\sqrt3$ is also needed.
Consequently, $\delta_0(\Lambda,6) = 2/\sqrt3$.
It can be easily checked that
$\delta_0(\Lambda,k)=\delta_0(\Lambda,6)= 2/\sqrt3$ 
for every $k \geq 6$.

\section{Concluding remarks} \label{sec:remarks}

We have given constructive upper bounds and derived close lower bounds 
on the degree 3 dilation of the infinite square lattice 
in the domain of plane geometric spanners. We have also derived exact
values for the degree $4$ dilation  of the square lattice along with
the degree 3 and $4$ dilation of the infinite hexagonal lattice. It is
easy to verify that our bounds also apply for finite sections 
of these lattices; see \cite{DG15b} for some examples.  

 It may be worth pointing out that in addition to the low stretch factors
achieved, the constructed spanners in this paper also have low weight and low
geometric dilation\footnote{When the stretch factor (or dilation)
  is measured over all pairs of points on edges or vertices of a plane graph $G$
  (rather than only over pairs of vertices) one arrives at the concept
  of \emph{geometric dilation} of $G$.}; see for instance~\cite{DEG+06,EGK06}
for basic terms. That is, each of these two parameters is at most a small constant factor
times the optimal one attainable.

As shown in Theorem~\ref{thm:square:3}, the degree $3$ dilation of the infinite
square lattice is at most $(3+2\sqrt2) \, 5^{-1/2}$.
It would be interesting to know whether this upper bound can be improved,
and so we put forward the following. 

\begin{conjecture}
  Let $\Lambda$ be the infinite square lattice. Then
  $\delta_0(\Lambda,3) =(3+2\sqrt2) \, 5^{-1/2} = 2.6065\ldots$
\end{conjecture}

\paragraph{A lighter degree $3$ spanner.}
The graph $G$ is illustrated in Fig.~\ref{fig:f22}. It is easy to
check that it is ``shorter'' than each of the two previous spanners of
degree $3$ for the square lattice analyzed in Section~\ref{sec:square}. 
\begin{figure}[ht]
\begin{center}
  \includegraphics[scale=0.43]{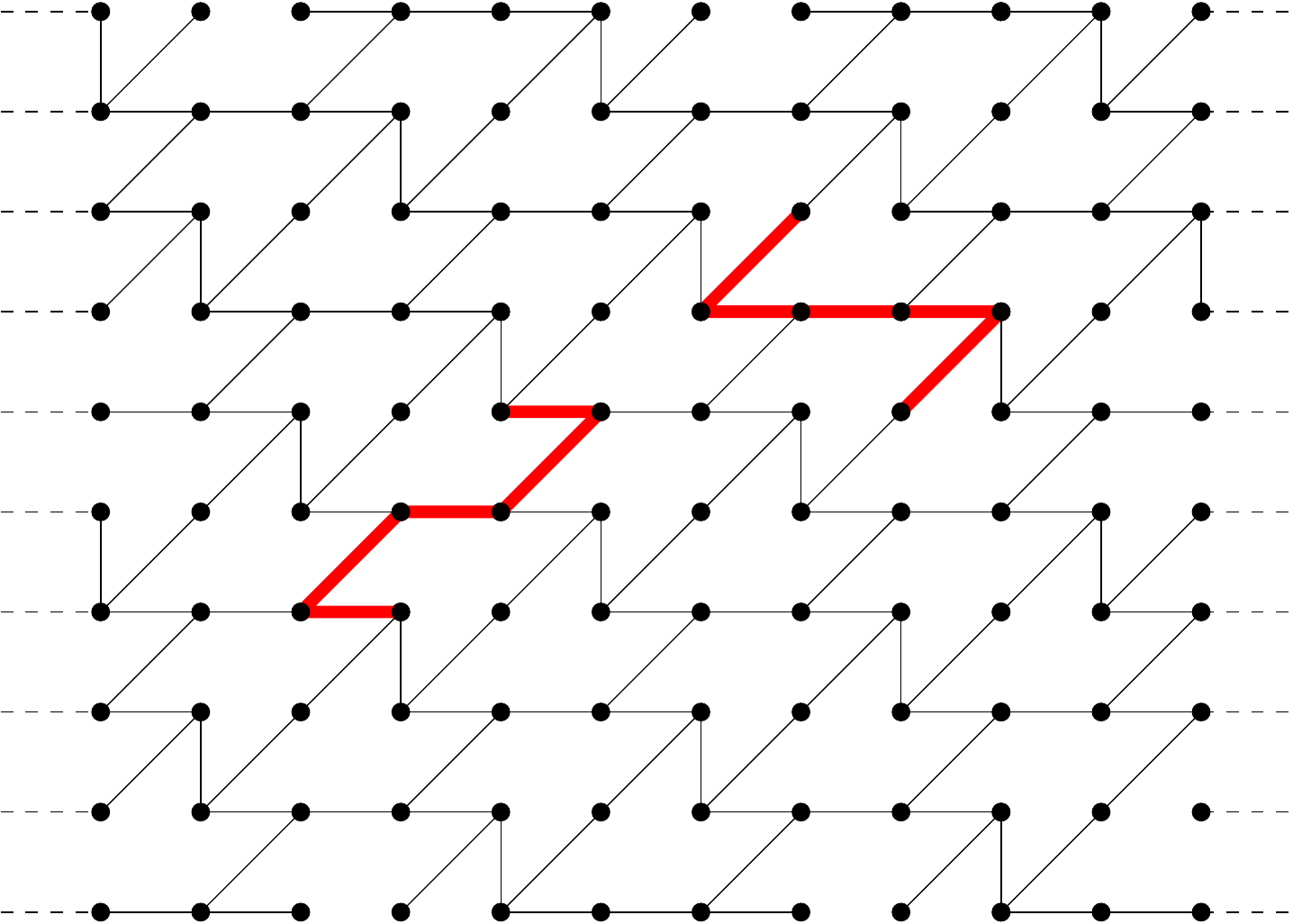}
\end{center}
\caption{A lighter degree $3$ spanner on the infinite square
  lattice. The shortest paths between point pairs with
  pairwise stretch factor $\delta_0$ are shown in red. } 
\label{fig:f22}
\end{figure}

Indeed, the average cost (length) per vertex is in this case smaller
(note that some vertices have degree $2$):
$$ \frac45 \left( \frac12 + \frac12 + \frac{\sqrt2}{2} \right)
+ \frac15 \left( \frac{\sqrt2}{2} + \frac{\sqrt2}{2} \right)
= \frac45 + \frac35 \sqrt2 = 1.6485\ldots, $$
while that for the previous spanners it is 
$\frac12 + \frac12 + \frac{\sqrt2}{2} =1.7071\ldots$ 
In particular, the total length of a square lattice section with $n$ points 
is $1.6485\ldots n + o(n)$ rather than $1.7071\ldots n + o(n)$.
If the stretch factor of $G$ would also be $\delta_0 = (3+2\sqrt2) \, 5^{-1/2}$, 
$G$ would be superior from the length perspective 
to the two spanners described in Section~\ref{sec:square}. 
We conjecture that the stretch factor of this lighter degree $3$ spanner
shown in Fig.~\ref{fig:f22} equals $\delta_0$.


\begin{thebibliography}{99}
\itemsep 2pt

\bibitem{AKK+08}
P.~K.~Agarwal, R.~Klein, C.~Kane, S.~Langerman, P.~Morin, M.~Sharir, and M.~Soss,
Computing the detour and spanning ratio of paths, trees, and cycles in 2D and 3D,
\emph{Discrete Comput. Geom.}
{\bf 39(1-3)}  (2008), 17--37.

\bibitem{ADD+93}
  I.~Alth\"ofer, G.~Das, D.~P.~Dobkin, D.~Joseph, and J.~Soares,
  On sparse spanners of weighted graphs,
  \emph{Discrete Comput. Geom.}
  {\bf 9} (1993), 81--100. 

\bibitem{ABC+08} B.~Aronov, M.~de~Berg, O.~Cheong, J.~Gudmundsson,
H.~J.~Haverkort, and A.~Vigneron,
Sparse geometric graphs with small dilation,
 \emph{Comput. Geom.}
{\bf 40(3)}  (2008), 207--219.

\bibitem{BGHP10} N.~Bonichon, C.~Gavoille, N.~Hanusse, and L.~Perkovi\'c,
  Plane spanners of maximum degree six,
  In \emph{Proc. Internat. Colloq. Automata, Lang. and Prog.},
  Springer, 2010, pp.~19--30. 

\bibitem{BKPX15} N.~Bonichon, I.~Kanj, L.~Perkovi\'c, and G.~Xia,
  There are plane spanners of degree 4 and moderate stretch factor,
  \emph{Discrete Comput. Geom.}
  {\bf 53(3)} (2015), 514--546. 

 \bibitem{BCC12} P.~Bose, P.~Carmi, and L.~Chaitman-Yerushalmi,
  On bounded degree plane strong geometric spanners,
  \emph{J. Discrete Algorithms}
  {\bf 15} (2012), 16--31. 

\bibitem{BGS05} P.~Bose, J.~Gudmundsson, and M.~Smid,
  Constructing plane spanners of bounded degree and low weight,
  {\em Algorithmica} {\bf 42} (2005), 249--264.
  
\bibitem{PBMS13} P.~Bose and M.~Smid,
  On plane geometric spanners: A survey and open problems,
   \emph{Comput. Geom.}
\textbf{46(7)} (2013), 818--830. 

\bibitem{BSX09} P.~Bose, M.~Smid, and D.~Xu,
  Delaunay and diamond triangulations contain spanners of bounded degree,
   {\em Internat. J. Comput. Geom. Appl.}
       {\bf 19(2)} (2009), 119--140. 

\bibitem{CDNS95}
  B.~Chandra, G.~Das, G.~Narasimhan, and J.~Soares,
  New sparseness results on graph spanners,
  {\em Internat. J. Comput. Geom. Appl.}
  {\bf 5} (1995), 125--144. 

\bibitem{CHL08} O.~Cheong, H.~Herman, and M.~Lee,
 Computing a minimum-dilation spanning tree is NP-hard,
   \emph{Comput. Geom.}
  {\bf 41(3)} (2008), 188--205.

\bibitem{Ch89} P.~Chew,
  There are planar graphs almost as good as the complete graph,
  \emph{J. of Computer and System Sci.}
       {\bf 39(2)} (1989), 205--219.
       
\bibitem{DH96} G.~Das and P.~Heffernan,
  Constructing degree-3 spanners with other sparseness properties,
\emph{Internat. J. Found. Comput. Sci.}
  {\bf 7(2)} (1996), 121--136. 

\bibitem{DEG+06}
A.~Dumitrescu, A.~Ebbers-Baumann, A.~Gr\"une, R.~Klein, and G.~Rote,
On the geometric dilation of closed curves, graphs, and point sets,
 \emph{Comput. Geom.}
{\bf 36} (2006), 16--38.

\bibitem{DG15a}
A.~Dumitrescu and A.~Ghosh,
Lower bounds on the dilation of plane spanners,
\emph{Internat. J. Comput. Geom. Appl.}, to appear.
Preprint, \url{arXiv:1509.07181}.

\bibitem{DG15b}
A.~Dumitrescu and A.~Ghosh,
Lattice spanners of low degree,
{\em Proc. Internat. Conf. Algor. and Discrete Appl. Math.},
2016, vol. 9602 of LNCS, pp. 152-163.
Preprint, \url{arXiv:1602.04381}.

\bibitem{EGK06}
  A.~Ebbers-Baumann, A.~Gr\"une, and R.~Klein,
  On the geometric dilation of finite point sets,
  {\em Algorithmica} {\bf 44} (2006), 137--149.

\bibitem{EKLL04}
  A.~Ebbers-Baumann, R.~Klein, E.~Langetepe, and A.~Lingas,
  A fast algorithm for approximating the detour of a polygonal chain,
   \emph{Comput. Geom.}
  {\bf 27} (2004), 123--134.

\bibitem{Epp00}
  D.~Eppstein, Spanning trees and spanners,
  in {\em Handbook of Computational Geometry} (J.~R.~Sack and J.~Urrutia, editors),
  North-Holland, Amsterdam, 2000, pp.~425--461.

\bibitem{GK07}
J.~Gudmundsson and C.~Knauer,
Dilation and detour in geometric networks,
in {\em Handbook on Approximation Algorithms and Metaheuristics, Chap.~52}
(T.~Gonzalez, editor), Chapman \& Hall/CRC, Boca Raton, 2007.

 \bibitem{KP08} I.~Kanj and L.~Perkovi\'c,
  On geometric spanners of Euclidean and unit disk graphs,
  in \emph{Proc. 25th Annual Sympos. on Theoretical Aspects of Comp. Sci.},
Schloss Dagstuhl Leibniz--Zentrum f\"ur Informatik, 2008, pp.~409--420. 

\bibitem{KG92}
  M.~Keil and C.~A.~Gutwin,
  Classes of graphs which approximate the complete Euclidean graph,
   \emph{Discrete Comput. Geom.}
   {\bf 7} (1992), 13--28. 

\bibitem{KKP15}
R. Klein, M. Kutz, and R. Penninger,
Most finite point sets in the plane have dilation {\textgreater} 1,
 \emph{Discrete Comput. Geom.}
 {\bf 53(1)} (2015), 80--106.

\bibitem{Le83}
  T. Leighton,
  \emph{Complexity Issues in VLSI, Foundations of Computing Series},
  MIT Press, Cambridge, MA, 1983.
 
\bibitem{LL92}
  C.~Levcopoulos and A.~Lingas,
  There are planar graphs almost as good as the complete graphs
  and almost as cheap as minimum spanning trees,
  {\em Algorithmica} {\bf 8} (1992), 251--256.

\bibitem{LW04} X.~Y.~Li and Y.~Wang,
  Efficient construction of low weight bounded degree planar spanner, 
   {\em Internat. J. Comput. Geom. Appl.}
       {\bf 14(1--2)} (2004), 69--84. 

\bibitem{LS93} A.~L.~Liestman and T.~C.~Shermer, Grid spanners,
\emph{Networks} \textbf{23(2)} (1993), 123--133.

 \bibitem{LSS96} A.~L.~Liestman, T.~C.~Shermer, and C.~R.~Stolte,
 Degree-constrained spanners for multidimensional grids,
 \emph{Discrete Appl. Math.} \textbf{68(1)} (1996), 119--144.

 \bibitem{Mu04}W.~Mulzer,
  Minimum dilation triangulations for the regular $n$-gon,
  Masters Thesis, Freie Universit\"{a}t Berlin, 2004. 

\bibitem{NS07} G.~Narasimhan and M.~Smid,
  {\em Geometric Spanner Networks},
  Cambridge University Press, 2007.
  
\bibitem{Smid-open} M.~Smid,
Progress on open problems in the book Geometric Spanner Networks,
\url{http://people.scs.carleton.ca/~michiel/SpannerBook/openproblems.html}.

\bibitem{Xia13} G.~Xia,
The stretch factor of the Delaunay triangulation is less than 1.998,
\emph{SIAM J. Comput.},
\textbf{42(4)} (2013), 1620--1659.

\end{thebibliography}
\end{document}